\documentclass[a4paper]{amsart}

\usepackage[utf8]{inputenc} 
\usepackage[british]{babel} 
\usepackage{culmus} 
\usepackage{dsfont}

\usepackage{hyperref}

\begin{hyphenrules}{british}
\end{hyphenrules}

\usepackage{graphicx} 
\usepackage{booktabs} 
\usepackage{array} 
\usepackage{paralist} 
\usepackage{verbatim} 
\usepackage{subfig} 
\usepackage{scalerel,stackengine} 
\usepackage{bmpsize}
\usepackage{amssymb}
\usepackage{mathtools}
\usepackage{amsmath}	
\usepackage{amssymb}
\usepackage{amsthm}
\usepackage[mathcal]{eucal}
\usepackage{graphicx}
\usepackage{faktor} 
\usepackage{mathrsfs} 
\usepackage[all, cmtip]{xy} 
\usepackage{systeme}
\usepackage[titletoc]{appendix}
\usepackage{leftindex}
\usepackage{enumitem}
\usepackage{appendix}
\usepackage{yhmath} 

\theoremstyle{plain}
\newtheorem{theorem}[subsection]{Theorem}
\newtheorem{lemma}[subsection]{Lemma}
\newtheorem{proposition}[subsection]{Proposition}
\newtheorem{corollary}[subsection]{Corollary}

\theoremstyle{definition}
\newtheorem{definition}[subsection]{Definition}

\theoremstyle{remark}

\newtheorem{examples}[subsection]{Examples}
\newtheorem{remark}[subsection]{Remark}
\newtheorem{remarks}[subsection]{Remarks}

\stackMath
\newcommand\reallywidehat[1]{%
\savestack{\tmpbox}{\stretchto{%
  \scaleto{%
    \scalerel*[\widthof{\ensuremath{#1}}]{\kern-.6pt\bigwedge\kern-.6pt}%
    {\rule[-\textheight/2]{1ex}{\textheight}}
  }{\textheight}%
}{0.5ex}}%
\stackon[1pt]{#1}{\tmpbox}%
}
\newcommand{\pullback}[4]{{#1 \leftindex_{#2}{\times}_{#3} #4}}
\newcommand{\rel}[1]{\mathcal{R}_{\mathbb{#1}}} 
\newcommand{\R}[1]{R_{\mathbb{E}}}
\renewcommand{\P}[1]{P_{\mathbb{E}}}
\newcommand{\Ch}[1]{\mathfrak{Ch}_{\mathbb{#1}}} 
\newcommand{\la}{\langle}
\newcommand{\ra}{\rangle}
\usepackage{color}

\DeclareSymbolFont{alephbet}{HE8}{frank}{m}{n}
\SetSymbolFont{alephbet}{bold}{HE8}{frank}{b}{n}
\DeclareMathSymbol{\samech}{\mathord}{alephbet}{"F1}
\DeclareMathSymbol{\mem}{\mathord}{alephbet}{"EE}

\usepackage{chngcntr}
\counterwithin*{equation}{section}

\def\pullback{
 \ar@{-}[]+R+<4pt,-3pt>;[]+RD+<4pt,-6pt>%
 \ar@{-}[]+D+<3pt,-6pt>;[]+RD+<4pt,-6pt>}

\def\pushout{%
  \ar@{-}[]+U+<-2pt,5pt>;[]+LU+<-4pt,5pt>
  \ar@{-}[]+L+<-4pt,2pt>;[]+LU+<-4pt,5pt>}

\begin{document}

\title[The direction of $\mathcal{R}$-full Schreier internal categories]{$\mathcal{R}$-full Schreier internal categories and their directions}

\author[S.~Ambra]{Stefano Ambra}
\address[Stefano Ambra]{Dipartimento di Matematica ``Federigo Enriques'', Universit\`{a} degli Studi di Milano, Via Saldini 50, 20133 Milano, Italy}
\email{stefano.ambra@unimi.it}

\author[A.~Montoli]{Andrea Montoli}
\address[Andrea Montoli]{Dipartimento di Matematica ``Federigo Enriques'', Universit\`{a} degli Studi di Milano, Via Saldini 50, 20133 Milano, Italy}
\email{andrea.montoli@unimi.it}

\author[D.~Rodelo]{Diana Rodelo}
\address[Diana Rodelo]{Department of Mathematics, University of the Algarve, 8005-139 Faro, Portugal and Center for Research and Development in Mathematics and Applications (CIDMA), Department of Mathematics, University of Aveiro, 3810-193 Aveiro, Portugal} \email{drodelo@ualg.pt}

\thanks{This work was supported by the Shota Rustaveli National Science Foundation of Georgia (SRNSFG), through grant FR-24-9660, ``Categorical methods for the study of cohomology theory of monoid-like structures: an approach through Schreier extensions''.
The first and second authors are members of the Gruppo Nazionale per le Strutture Algebriche, Geometriche e le loro Applicazioni (GNSAGA) dell'Istituto Nazionale di Alta Matematica ``Francesco Severi''.
The third author acknowledges financial support by CIDMA (https://ror.org/05pm2mw36)
under the Portuguese Foundation for Science and Technology
(FCT, https://ror.org/00snfqn58), Grants UID/04106/2025 (https://doi.org/10.54499/UID/04106/2025)
and UID/PRR/04106/2025 (https://doi.org/10.54499/UID/PRR/04106/2025).}

\begin{abstract}
We introduce the notion of \emph{$\mathcal{R}$-full Schreier internal category}, which is the monoid analogue of the notion of aspherical abelian groupoid. We associate with every $\mathcal{R}$-full Schreier internal category a \emph{direction}, which is a Schreier split extension with commutative and cancellative kernel. We show that this association is functorial and that such functor is a product preserving cofibration. Thanks to these properties, we equip the connected components of the fibres of such functor with canonical commutative monoid structures. Using the equivalence between Schreier internal categories and crossed semimodules of monoids, we describe these commutative monoids in terms of crossed Schreier extensions.
\end{abstract}



\maketitle
\section{Introduction}
The classical Eilenberg-Mac Lane cohomology of groups \cite{EML} is known to admit an interpretation in terms of group extensions. Already in \cite{EML} it was shown that the second cohomology group $H^2(G, A, \varphi)$ of a group $G$ with coefficients in a $G$-module $(A, \varphi)$ is isomorphic to the abelian group $OpExt(G, A, \varphi)$ of isomorphism classes of extensions of $G$ by $A$ inducing the action $\varphi.$ Interpretations of higher cohomology groups $H^n(G, A, \varphi)$ in terms of suitable long extensions have been obtained in \cite{Holt, Huebschmann}. In particular, the elements of the third cohomology group $H^3(G, A, \varphi)$ correspond to equivalence classes of crossed extensions, namely exact sequences of groups
\[ \xymatrix{
A \ar@{>->}[r]^-k &B \ar[r]^-{\delta} &C \ar@{->>}[r]^-{\pi} &G
} \]
where the morphism $\delta,$ together with an action $\alpha$ of $C$ on $B,$ forms a crossed module, and such that the induced $G$-action on $A$ coincides with the given action $\varphi.$

These results can be described and understood categorically by means of the so-called \emph{direction functors}. The ($0$-dimensional) direction functor was introduced in \cite{bourn-direction}: it is a functor $d\colon \mathrm{AutMal}({\mathcal{C}}_g)\to\mathbf{Ab}(\mathcal{C})$, where $\mathcal{C}$ is a Barr-exact category~\cite{barr}, $\mathrm{AutMal}({\mathcal{C}}_g)$ is the category of objects with global support equipped with an autonomous Mal'tsev operation, and $\mathbf{Ab}(\mathcal{C})$ is the category of internal abelian groups in $\mathcal{C}$. Due to the good properties of $d$ (which is a conservative, finite product preserving (pseudo)-cofibration), given any $A\in \mathbf{Ab}(\mathcal{C})$, there is an induced abelian group structure on the connected components of the fibre $d^{-1}(A)$. In particular, when $\mathcal{C}=\mathbf{Gp}/G$, where $\mathbf{Gp}$ is the category of groups, $\mathrm{AutMal}({\mathcal{C}}_g)$ is the category of extensions over $G$ with abelian kernel, $\mathbf{Ab}(\mathcal{C})$ is the category of $G$-modules, and $d(\xymatrix@R=10pt{A \ar@{>->}[r]^-k & X \ar@{>>}[r]^f & G})=(A,\ast)$, where $\ast$ is the $G$-action on $A$ given by conjugation in $X$. As a consequence, the abelian group structures of all $G$-modules $(A,\varphi)$ can be lifted to their fibres, which correspond precisely to categories of extensions of $G$ with abelian kernel; thus the abelian group of connected components of $d^{-1}(A,\varphi)$ is isomorphic to $H^2(G,A,\varphi)$.

The second step in the categorical description of cohomology in terms of extensions was accomplished in \cite{aspherical}, where the $1$-dimensional direction functor was introduced: it is now a functor $d_1\colon \mathrm{AAGpd}(\mathcal{C})\to\mathbf{Ab}(\mathcal{C}),$ where $\mathrm{AAGpd}(\mathcal{C})$ is the category of aspherical abelian groupoids in $\mathcal{C}.$ As in the 0-dimensional case, given any $A\in \mathbf{Ab}(\mathcal{C})$, there is an induced abelian group structure on the set of connected components of $d_1^{-1}(A)$. Using the classical equivalence between internal groupoids in the category of groups and crossed modules (established in \cite{BS}), and again by the good properties of the direction functor $d_1$ when $\mathcal{C}=\mathbf{Gp}/G$, one obtains the canonical description of the abelian group structure on the set of equivalence classes of crossed extensions inducing the same action. Thus, given a $G$-module $(A,\varphi)$, the abelian group of connected components of $d_1^{-1}(A,\varphi)$ is isomorphic to $H^3(G,A,\varphi)$. This procedure can be iterated in the higher dimensions, giving a complete description of cohomology in terms of suitable long extensions (see \cite{Bourn-Rodelo}).

In \cite{AMR}, we constructed a variation of the $0$-dimensional direction functor, resulting in a new functor $d,$ suitable to give a categorical description of low-dimensional cohomology of monoids in terms of appropriate extensions. The domain of $d$ is the category $cc\text{-}SExt_M$ whose objects are the Schreier extensions (in the sense of \cite{Redei}) of a monoid $M,$ with commutative and cancellative kernel, and the codomain is the category of internal commutative monoids in $cc\text{-}SExt_M$ (which is nothing but the category of Schreier split extensions of $M,$ with commutative and cancellative kernel). The main difference, with respect to the construction of the direction functor in \cite{bourn-direction}, is that, for an object $f$ in $cc\text{-}SExt_M$, instead of considering the kernel equivalence relation $\mathcal{E}q(f)$ of $f$ we consider a reflexive and transitive (but not necessarily symmetric) subrelation $\mathcal{R}$ of $\mathcal{E}q(f)$. Still, as it is shown in \cite{AMR}, the functor $d$ retains the good properties which are needed to lift the commutative monoid structure of any object in its codomain to the fibres, equipping their connected components with canonical commutative monoid structures. The commutative monoids obtained through this approach are isomorphic to the second cohomology monoids of the cohomology theory introduced by Patchkoria in \cite{Patchkoria77}.

The present paper continues the work done in \cite{AMR}, extending to the case of monoids the construction of the $1$-dimensional direction functor of \cite{aspherical, Bourn-Rodelo}. The main ingredient here is what we call $\mathcal{R}$-full Schreier internal categories, defined by means of the same subrelation of the kernel pair that is used in the construction of $d.$ In the case of monoids, $\mathcal{R}$-full Schreier internal categories play the role which, in the case of groups, is played by the aspherical abelian groupoids of \cite{aspherical}. Using the equivalence, proved in \cite{P-semi}, between Schreier internal categories and crossed semimodules (that are the monoid analogue of crossed modules of groups), we show that the connected components of the fibres of the new $1$-dimensional direction functor $d_1$ can be described as sets of equivalence classes of suitable crossed extensions of monoids. Moreover, $d_1$ is shown to satisfy the same good properties of the $1$-dimensional direction functor considered in \cite{aspherical, Bourn-Rodelo}, a fact which allows, once again, to lift the commutative monoid structure of any object in the codomain of $d_1$ to the fibres, and to equip the connected components of these fibres with canonical commutative monoid structures. This generalizes to monoids the results of \cite{Holt, Huebschmann} on the characterization of $3$-dimensional cohomology of groups in terms of crossed extensions.

\section{Crossed semimodules and crossed Schreier extensions of monoids}
\label{sec:crossed_semimodules}
Recall that, classically, a \emph{crossed module} \cite{whitehead_1, whitehead_2} is a pair $(B\xlongrightarrow{\delta} C,\alpha)$ where $B$ and $C$ are groups, $\delta$ is a group homomorphism and $\alpha$ is a group action of $C$ on $B$ (viz., a group homomorphism $\alpha\colon C\rightarrow\mathrm{Aut}(B),$ or equivalently a map of sets $\alpha\colon C\times B\dashrightarrow B,$ $(c,b)\mapsto c\ast b,$ with the usual axioms) satisfying the following conditions:
\begin{enumerate}
\item $\delta(c\ast b)=c\cdot\delta(b)\cdot c^ {-1}$ for all $b\in B,$ $c\in C$ \ (\emph{pre-crossed module condition});
\item $\delta(b_1)\ast b_2=b_1\cdot b_2\cdot{b_1}^{-1}$ for all $b_1,b_2\in B$ \ (\emph{Peiffer condition}).
\end{enumerate}
Morphisms of crossed modules are defined accordingly, resulting in an equivalence $\mathbf{Xmod}\cong\mathrm{Gpd}(\mathbf{Gp})$ between the category of crossed modules and the category of (internal) groupoids in the category of groups (\cite{BS}).

An exact sequence
\begin{equation}
\label{eqn:crossed_sequence}
\xymatrix{
A \ar@{>->}[r]^-k &B \ar[r]^-{\delta} &C \ar@{->>}[r]^-{\pi} &G
}
\end{equation}
in $\mathbf{Gp}$ such that $(B\xlongrightarrow{\delta} C,\alpha)$ is a crossed module is called a \emph{crossed extension}.

It is not difficult to prove that, given a crossed extension~\eqref{eqn:crossed_sequence}, the kernel $k$ is a central group  homomorphism (so that the group $A$ is abelian), and that the $C$-action $\alpha$ on $B$ restricts to a group action of $C$ on $A;$ as a consequence, there results a group action of $G$ on $A,$ (well) defined by
\begin{equation}
\label{eqn:induced_action_crossed_sequence}
G\rightarrow \mathrm{Aut}(A), \ \ g=\pi(c)\mapsto \alpha(c).
\end{equation}
When $G$ is a group and $\big(A,\omega\colon G\rightarrow\mathrm{Aut}(A)\big)$ is a $G$-module, a crossed extension~\eqref{eqn:crossed_sequence} is called a crossed extension of $G$ by the $G$-module $(A,\omega)$ if the induced action~\eqref{eqn:induced_action_crossed_sequence} coincides with the given action $\omega.$

Then one can show that the classical Eilenberg-Mac Lane third cohomology group $H^3(G,A,\omega)$ \cite{EML} of the (non necessarily abelian) group $G$ with coefficients in the $G$-module $(A,\omega)$ admits a description in terms of (equivalence classes of) crossed extensions of $G$ by $(A,\omega),$ see \cite{Holt, Huebschmann}.

The definition of a crossed module can be generalised to monoids as follows, as it was first observed in \cite{JS}.
\begin{definition}
\label{def:crossed_semimodule}
A \emph{crossed semimodule} is a pair $(B\xlongrightarrow{\delta} C,\alpha)$ where $B$ and $C$ are monoids, $\delta$ is a monoid homomorphism and $\alpha$ is a monoid action of $C$ on $B$ (viz., a monoid homomorphism $\alpha\colon C\rightarrow\mathrm{End}(B),$ or equivalently a map of sets $\alpha\colon C\times B\dashrightarrow B,$ $(c,b)\mapsto c\ast b,$ with the usual axioms) satisfying the following conditions:
\begin{enumerate}
\item $\delta(c\ast b)+c=c+\delta(b)$ \ for all $b\in B,$ $c\in C;$
\item $\delta(b_1)\ast b_2+b_1=b_1+b_2$ \ for all $b_1,b_2\in B.$
\end{enumerate}
(Here, we denote by $+$ the monoid operation on $B$ and $C,$ with neutral element $0,$ though neither are assumed to be commutative.)
\end{definition}
A morphism of crossed semimodules is then a pair of monoid homomorphisms $\varphi,\psi$ such that the diagrams
\begin{equation}
\label{eqn:mor_crossed_semimod}
\begin{aligned}
\xymatrix{
B \ar[d]_-{\varphi} \ar[r]^-{\delta} &C \ar[d]^-{\psi} \\
{B'} \ar[r]_-{\delta'} &{C',}
}
\
\
\xymatrix{
{C\times B} \ar[d]_-{\psi\times\varphi} \ar@{-->}[r]^-{\alpha} &B \ar[d]^-{\varphi} \\
{C'\times B'} \ar@{-->}[r]_-{\alpha'} &{B'}
}
\end{aligned}
\end{equation}
commute (in the categories $\mathbf{Mon}$ of monoids and $\mathbf{Set}$ of sets, respectively).

We end up with a category $\mathbf{Xsmod}$ of crossed semimodules in $\mathbf{Mon},$ and the equivalence $\mathbf{Xmod}\cong \mathrm{Gpd}(\mathbf{Gp})$ is generalised accordingly:
\begin{theorem}[\cite{P-semi}]
\label{thm:equivalence_crossed_semimodules_schreier_categories}
There is an equivalence $\mathbf{Xsmod}\cong\mathrm{SCat}(\mathbf{Mon})$ between the categories of crossed semimodules and of Schreier internal categories in $\mathbf{Mon},$ i.e. internal categories
\begin{equation}
\label{eqn:schreier_internal_category}
\xymatrix{{\underline{X}: X_1\times_{X_0}X_1} \ar[r]^-{m_X} &{X_1} \ar@<.9ex>[r]^-{d_X} \ar@<-.9ex>[r]_-{c_X} &{X_0} \ar[l]|{e_X} }
\end{equation}
such that every $x\in X_1$ can be written as $x=k_d(a)+e_Xd_X(x)$ for a unique $a\in K[d_X],$ where $k_d\colon K[d_X]\rightarrowtail X_1$ is a kernel of $d_X.$ (In the terminology of \cite{MMS}, Definition 2.6, this property is expressed by saying that $(d_X,e_X)$ is a \emph{Schreier point}.)
\end{theorem}
The equivalence, at the level of objects, is obtained as follows:
\begin{itemize}
\item a crossed semimodule $(B\xlongrightarrow{\delta} C,\alpha)$ is mapped to the Schreier (internal) category
\begin{equation*}
\xymatrixcolsep{2.5pc}
\xymatrix{{(B\rtimes_{\alpha} C)\times_C (B\rtimes_{\alpha} C)} \ar[r]^-m &{B\rtimes_{\alpha} C} \ar@<.9ex>[r]^-{p} \ar@<-.9ex>[r]_-{\lambda} &{C,} \ar[l]|-s}
\end{equation*}
where $p(b,c)=c,$ $s(c)=(0,c),$ $\lambda(b,c)=\delta(b)+c$ and $m(b_1,b_2,c)=(b_2+b_1,c);$
\item conversely, a Schreier internal category~\eqref{eqn:schreier_internal_category} is mapped to the crossed semimodule
\begin{equation*}
(K[d_X]\xlongrightarrow{c_Xk_d} X_0,\alpha),
\end{equation*}
where $\alpha\colon X_0\times K[d_X]\dashrightarrow K[d_X],$ $(x,a)\mapsto a'$ such that $e_X(x)+k_d(a)=k_d(a')+e_X(x),$ is the action associated with the Schreier point $(d_X,e_X).$
\end{itemize}

Under the previous equivalence, Schreier internal groupoids in $\mathbf{Mon}$ correspond to the crossed semimodules $(B\xlongrightarrow{\delta} C,\alpha)$ such that $B$ is a group, see \cite{MMS}.

Now, consider a sequence
\begin{equation}
\label{eqn:crossed_sequence_monoids}
\xymatrix{
A \ar@{>->}[r]^-k &B \ar[r]^-{\delta} &C \ar@{->>}[r]^-{\pi} &M
}
\end{equation}
in $\mathbf{Mon}$, where $(B\xlongrightarrow{\delta} C,\alpha)$ is a crossed semimodule, $k$ is a kernel of $\delta$, and $\pi$ is a cokernel of $\delta$. As in the case of groups, it is easy to see that $k$ is central, but it is no longer true that $\alpha$ restricts to an action of $C$ on $A.$ It is observed in \cite{AMMPS} that this property is restored if some stronger assumptions hold for the sequence~\eqref{eqn:crossed_sequence_monoids}.

Indeed, recall that a short exact sequence of monoids and monoid homomorphisms $\xymatrix{K \ar@{>->}[r]^-{k} &X \ar@{->>}[r]^-{f} &M}$ is a \emph{Schreier extension} \cite{Redei} if for every $m\in M$ there exists $u_m\in f^{-1}(m)$ such that any $x\in f^{-1}(m)$ decomposes as $x=k(a)+u_m$ for a unique $a\in K,$ depending on $x$ and on $u_m.$ Since this property does not depend on the choice of the kernel $(K,k),$ we shall often simply say that $f$ is a Schreier extension, if the latter holds. The element $u_m$ is called a \emph{representative} of $m,$ and it is easily checked that in this case another element $v_m=k(a)+u_m$ in the fibre $f^{-1}(m)$ is a representative if and only if $a$ is invertible in $K.$

One can prove that Schreier extensions are pullback stable, meaning that if $\xymatrix{{f\colon X} \ar@{->>}[r] &M}$ is a Schreier extension and
\begin{equation*}
\xymatrix{
P \pullback\ar@{->>}[d]_-{\overline{f}} \ar[r]^-{\overline{g}} &X \ar@{->>}[d]^-f \\
Y \ar[r]_g &M
}
\end{equation*}
is any pullback in $\mathbf{Mon},$ then $\overline{f}$ is also a Schreier extension (see \cite{AMMPS}).

Consider then a (regular epimorphism, monomorphism) factorisation of $\delta$ in the sequence \eqref{eqn:crossed_sequence_monoids}:
\begin{equation}
\label{eqn:factorised_crossed_semimodule}
\begin{aligned}
\xymatrix{
A \ar@{>->}[r]^-k &B \ar@{->>}[rd]_-{\rho} \ar[rr]^-{\delta} & \ &C \ar@{->>}[r]^-{\pi} &{M.} \\
 &\ &N \ar@{>->}[ru]_-{\nu}&\ &\ &\
}
\end{aligned}
\end{equation}
We shall say that:
\begin{itemize}
\item the condition $\mathbf{CS}_r$ holds for the sequence \eqref{eqn:crossed_sequence_monoids}, if in diagram \eqref{eqn:factorised_crossed_semimodule} the monoid $N$ is right cancellative and $\xymatrix{N \ar@{>->}[r]^-{\nu} &C \ar@{->>}[r]^-{\pi} &M}$ is a Schreier extension of monoids;
\item the condition $\mathbf{CS}_l$ holds for the sequence \eqref{eqn:crossed_sequence_monoids}, if in diagram~\eqref{eqn:factorised_crossed_semimodule} the monoid $A$ is right cancellative and $\xymatrix{A \ar@{>->}[r]^-{k} &B \ar@{->>}[r]^-{\rho} &N}$ is a Schreier extension of monoids.
\end{itemize}
Then the following is proven in \cite{AMMPS}:
\begin{itemize}
\item if $\mathbf{CS}_r$ holds, the $C$-action $\alpha$ on $B$ restricts to an action of $C$ on $A;$
\item if $\mathbf{CS}_r$ and $\mathbf{CS}_l$ hold, there results an action of $M$ on $A,$ (well) defined by $M\longrightarrow\mathrm{End}(A),$ $m=\pi(c)\mapsto \alpha(c).$
\end{itemize}

\begin{definition}[\cite{AMMPS}]
\label{def:crossed_Schreier_extension}
A sequence \eqref{eqn:crossed_sequence_monoids} is called a \emph{crossed Schreier extension} when both $\mathbf{CS}_r$ and $\mathbf{CS}_l$ hold.
\end{definition}

Our first goal, in the next section, will be to study the meaning of the conditions $\mathbf{CS}_r$ and $\mathbf{CS}_l$ for the Schreier internal category corresponding (by Theorem~\ref{thm:equivalence_crossed_semimodules_schreier_categories}) to the crossed semimodule $\delta$ in~\eqref{eqn:crossed_sequence_monoids}, but before delving into this discussion, we find it convenient to provide some explicit examples showing how these conditions may or may not hold.
\begin{examples}
\begin{enumerate}
\item Consider the natural order $\mathcal{O}_{\mathbb{N}}=\{(x,y)\in\mathbb{N}\times\mathbb{N}:x\leq y\}$ on the monoid $(\mathbb{N},+,0)$ of natural numbers. Then, if $T_{\mathbb{N}}=\{(x,y,z)\in\mathbb{N}\times\mathbb{N}\times\mathbb{N}:x\leq y\leq z\},$ $p_{1,3}$ denotes the projection $(x,y,z)\mapsto (x,z)$ (corresponding to the transitivity of $\leq$), and $p_1,$ $p_2$ and $\Delta_{\mathbb{N}}$ are the projections $(x,y)\mapsto x$ on the first factor, $(x,y)\mapsto y$ on the second factor, and the diagonal morphism $x\mapsto (x,x)$ respectively, the category
\begin{equation*}
\xymatrix{{T_{\mathbb{N}}} \ar[r]^-{p_{1,3}} &{\mathcal{O}_{\mathbb{N}}} \ar@<.9ex>[r]^-{p_1} \ar@<-.9ex>[r]_-{p_2} &{\mathbb{N}} \ar[l]|{\Delta_{\mathbb{N}}} }
\end{equation*}
is a Schreier category in $\mathbf{Mon}$ (see \cite[Example 3.1.2]{schreier_book}). We have in this case $K[p_1]\cong \mathbb{N},$ and the corresponding crossed semimodule is just the identity morphism on $\mathbb{N},$ so that both $\xymatrix{0 \ar[r] &{\mathbb{N}} \ar@{=}[r] &{\mathbb{N}}}$ and $\xymatrix{{\mathbb{N}} \ar@{=}[r] &{\mathbb{N}} \ar[r] &0}$ are Schreier extensions with cancellative kernels.
\item Now let $\mathcal{O}_{\mathbb{Z}}=\{(x,y)\in\mathbb{Z}\times\mathbb{Z}:x\leq y\}$ denote the natural order on $\mathbb{Z}.$ Then
\begin{equation*}
\xymatrix{{T_{\mathbb{Z}}} \ar[r]^-{p_{1,3}} &{\mathcal{O}_{\mathbb{Z}}} \ar@<.9ex>[r]^-{p_1} \ar@<-.9ex>[r]_-{p_2} &{\mathbb{Z}} \ar[l]|{\Delta_{\mathbb{Z}}} }
\end{equation*}
is again a Schreier category in $\mathbf{Mon}$ (\cite[Example 3.1.3]{schreier_book}), and we have $K[p_1]\cong\mathbb{N}.$ The corresponding crossed semimodule is the inclusion morphism $i\colon\mathbb{N}\rightarrowtail\mathbb{Z},$ and since $coker(i)=0,$ the sequence $\xymatrix{ {\mathbb{N}} \ar@{>->}[r]^-i &{\mathbb{Z}} \ar[r] &0}$ is in this case not even exact. Notice that, nonetheless, there is a trivial action of $M=0$ on $A=0=ker(i),$ so that the conditions $\mathbf{CS}_r,$ $\mathbf{CS}_l$ are certainly not necessary for a sequence~\eqref{eqn:crossed_sequence_monoids} to yield an induced action of $M$ on $A.$
\item If $L,$ $M$ are monoids and $L$ is commutative, any central monoid homomorphism $f\colon L\rightarrow M$ is a crossed semimodule with respect to the trivial action $m\ast l=l$ for any $m\in M$ and $l\in L$ (cf.~\cite[Example 3.3.3]{schreier_book}). Thus, in particular, any morphism in the category $\mathbf{CMon}$ of commutative monoids is trivially a crossed semimodule. The resulting sequence
\begin{equation*}
\xymatrix{
{ker(f)} \ar@{>->}[r]^-k &L \ar@{->>}[rd]_-{\rho} \ar[rr]^-f &\ &M \ar@{->>}[r]^-r &{coker(f)} \\
&\ &N \ar@{>->}[ru]_-{\nu} &\ &\
}
\end{equation*}
may or may not satisfy the conditions $\mathbf{CS}_r,$ $\mathbf{CS}_l,$ even if $f$ admits a (normal epimorphism, normal monomorphism) factorisation (i.e., even if $(k,\rho)$ and $(\nu,r)$ are indeed short exact sequences in $\mathbf{Mon}$).

For example, if $(M_2=\{0,1\},\cdot,1)$ is the commutative monoid with neutral element $1$ and $0\cdot 0=0,$ the monoid homomorphism $f\colon(\mathbb{N},\cdot,1)\rightarrow(M_2,\cdot,1)$ given by $f(n)=1$ if $n\neq0$ and $f(0)=0$ is a cokernel of its kernel $i\colon(\mathbb{N}\setminus\{0\},\cdot,1)\rightarrowtail(\mathbb{N},\cdot,1),$ $i(n)=n,$ but the sequence
\begin{equation*}
\xymatrix{{\mathbb{N}\setminus\{0\}}  \ar@{>->}[r]^-i &{\mathbb{N}} \ar@{->>}[r]^-f &{M_2}}
\end{equation*}
is not a Schreier one, so that $\mathbf{CS}_l$ fails (see \cite[Example 2.10.6]{AMR}), whereas for the unique monoid homomorphism $f\colon(\mathbb{N},+,0)\rightarrow C_m(t)$ defined by $f(1)=t$ (where $C_m(t)$ is the multiplicative cyclic group of order $m>1,$ with generator $t$), both $\mathbf{CS}_r$ and $\mathbf{CS}_l$ hold (see \cite[Example 4.2]{P-II}).
\end{enumerate}
\end{examples}
\section{$\mathcal{R}$-full Schreier internal categories in $cc\text{-}SExt_M$}
Using the notation introduced in \cite[Definition 3.2]{AMR}, by a \emph{$cc$-Schreier extension} we shall mean a Schreier extension of monoids
\begin{equation}
\label{eqn:schreier_extension}
\xymatrix{{\mathbb{E}:K} \ar@{>->}[r]^-{k} &X \ar@{->>}[r]^-{f} &M}
\end{equation}
whose kernel $K$ is commutative and cancellative. The category of $cc$-Schreier extensions with fixed codomain $M$ shall be denoted by $cc\text{-}SExt_M,$ the morphisms $\mathbb{E}\rightarrow\mathbb{E}'$ being the monoid homomorphisms $h$ making the diagram
\begin{equation*}
\xymatrix{
X \ar@{->>}[rd]_-f \ar[rr]^-h &\ &{X'} \ar@{->>}[ld]^-{f'} \\
&M
}
\end{equation*}
commute, and preserving the representatives in the sense that if $u_m$ is any representative in $\mathbb{E},$ then $h(u_m)$ is a representative in $\mathbb{E}'.$ It is not difficult to prove that the latter holds as soon as \emph{some} representative of $m$ in $\mathbb{E}$ is mapped by $h$ to a representative in $\mathbb{E}',$ see \cite[Proposition 2.8]{AMR}.

Recall also from \cite[Section 5]{AMR}, that to any $cc$-Schreier extension~\eqref{eqn:schreier_extension} one associates an internal reflexive and transitive relation
\begin{equation*}
\xymatrix{ {\rel{E}:R_{\mathbb{E}}}\ar@<.9ex>[r]^-{r_1} \ar@<-.9ex>[r]_-{r_2} &X \ar[l]|-{s_0}}
\end{equation*}
on $X,$ defined by $(x,y)\in R_{\mathbb{E}}$ if and only if $y=k(a)+x$ for some $a\in K$ (the element $a$ is then necessarily unique, by the cancellativity of $K$), with projections $r_1(x,y)=x,$ $r_2(x,y)=y$ and reflexivity given by the diagonal $s_0(x)=(x,x).$

Observe that there is a trivial inclusion $j\colon R_{\mathbb{E}}\rightarrowtail \mathrm{Eq}(f)=\{(x,y)\in X\times X:f(x)=f(y)\}$ by which $\rel{E}$ is a subobject of the kernel pair relation $\mathcal{E}q(f)$ of $f.$ Then one can prove that $\mathcal{E}q(f)$ is the congruence generated by $\rel{E},$ meaning that $f$ is a coequaliser of $r_1$ and $r_2$ (see \cite[Proposition 5.1]{AMR}), and that $\rel{E}$ is symmetric if and only if $K$ is a group (see \cite[Corollary 5.3]{AMR}).

Finally, the cancellativity of $K$ guarantees that the sequence
\begin{equation*}
\xymatrix{{K} \ar@{>->}[r]^-{\la0,k\ra} &{R_{\mathbb{E}}} \ar@{->>}[r]^-{r_1} &X}
\end{equation*}
is also a Schreier extension of monoids (see \cite[Proposition 5.5]{AMR}).

Now, consider a Schreier internal category
\begin{equation}
\label{eqn:schreier_category_SExt_M_with_Es}
\xymatrix{
{{{\mathbb{E}}_1}\times_{{\mathbb{E}}_0}{{\mathbb{E}}_1}} \ar[r]^-{m_X} &{{\mathbb{E}}_1} \ar@<.9ex>[r]^-{d_X} \ar@<-.9ex>[r]_-{c_X} &{{\mathbb{E}}_0} \ar[l]|-{e_X}
}
\end{equation}
in $cc\text{-}SExt_M,$ where $\mathbb{E}_0$ and $\mathbb{E}_1$ are the $cc$-Schreier extensions
\begin{equation*}
\begin{aligned}
{\xymatrix{{\mathbb{E}_0:K_0} \ar@{>->}[r]^-{k_0} &{X_0} \ar@{->>}[r]^-{f_0} &{M,}}}
\
\
{\xymatrix{{\mathbb{E}_1:K_1} \ar@{>->}[r]^-{k_1} &{X_1} \ar@{->>}[r]^-{f_1} &{M.}}}
\end{aligned}
\end{equation*}
We have therefore a commutative diagram of monoid homomorphisms
\begin{equation}
\begin{aligned}
\label{eqn:schreier_category_SExt_M}
\xymatrix{
{X_1} \ar@{->>}[rd]_-{f_1} \ar@<.9ex>[rr]^-{d_X} \ar@<-.9ex>[rr]_-{c_X} &\ &{X_0} \ar@{->>}[ld]^-{f_0} \ar[ll]|-{e_X} \\
&M
}
\end{aligned}
\end{equation}
such that $\xymatrix{{\underline{X}: X_1\times_{X_0}X_1} \ar[r]^-{m_X} &{X_1} \ar@<.9ex>[r]^-{d_X} \ar@<-.9ex>[r]_-{c_X} &{X_0} \ar[l]|-{e_X}}$ is a Schreier internal category in $\mathbf{Mon},$ and $d_X,$ $c_X$ and $e_X$ are morphisms of Schreier extensions.

The commutativity of~\eqref{eqn:schreier_category_SExt_M} means that we have equalities $f_1e_X=f_0$ and $f_0d_X=f_1=f_0c_X,$ which imply that there is an induced monoid homomorphism $\la d_X,c_X\ra$ from $X_1$ into the kernel pair of $f_0$:
\begin{equation*}
\xymatrix{
{X_1} \ar@/^1pc/[rrd]^-{c_X} \ar@/_1pc/[rdd]_-{d_X} \ar@{.>}[rd]|{\la d_X,c_X\ra} &\ &\ \\
&{\mathrm{Eq}(f_0)} \pullback\ar[d] \ar[r]  &{X_0} \ar@{->>}[d]^-{f_0} \\
&{X_0} \ar@{->>}[r]_-{f_0} &{M.}
}
\end{equation*}
\begin{proposition}
\label{prop:d_X,c_X-inside_R}
Denote by $k_d\colon K[d_X]\rightarrowtail X_1$ a kernel of $d_X,$ so that the crossed semimodule associated with $\underline{X}$ (as in Theorem~\ref{thm:equivalence_crossed_semimodules_schreier_categories}) is $\delta=c_Xk_d\colon K[d_X]\rightarrow X_0.$ Then:
\begin{enumerate}
\item the morphism $\la d_X,c_X\ra$ factorises through $j_0\colon R_{\mathbb{E}_0}\rightarrowtail \mathrm{Eq}(f_0);$
\item if $\la d_X,c_X\ra\colon X_1\rightarrow R_{\mathbb{E}_0}$ is epimorphic, $f_0$ is a cokernel of $\delta.$
\end{enumerate}
\end{proposition}
\begin{proof}
By Theorem~\ref{thm:equivalence_crossed_semimodules_schreier_categories}, there is an isomorphism $\xi$ making the diagram
\begin{equation*}
\xymatrixrowsep{1pc}
\xymatrix{
{X_1} \ar[dd]_-{\xi}^-{\wr} \ar@<.9ex>[r]^-{d_X} \ar@<-.9ex>[r]_-{c_X} &{X_0} \ar@{=}[dd] \ar@{->>}[rd]^-{f_0} \ar[l]|-{e_X} \\
&\ &M \\
{K[d_X]\rtimes_{\alpha}X_0} \ar@<.9ex>[r]^-{p} \ar@<-.9ex>[r]_-{\lambda} &{X_0} \ar[l]|-{s} \ar@{->>}[ru]_-{f_0}
}
\end{equation*}
commute in $\mathbf{Mon},$ where $\alpha$ is the action associated with the Schreier point $(d_X,e_X).$ Thus  $\la d_X,c_X\ra=\la p,\lambda\ra\xi.$ Since we have $\lambda(a,x)=c_Xk_d(a)+x$ and $f_0c_X\big(k_d(a)\big)$ $=$ $f_0d_X\big(k_d(a)\big)$ $=$ $0,$ it follows that indeed $\la p,\lambda \ra(a,x)=\big(x,c_Xk_d(a)+x\big)\in R_{\mathbb{E}_0},$ proving $(1).$

To prove $(2),$ observe first that since by assumption $(d_X,e_X)$ is a Schreier point, $d_X$ is a cokernel of its kernel $k_d$ (cf.~\cite[Section 4]{P-II} and \cite[Proposition 2.6]{AMR}). Then in the commutative diagram
\begin{equation*}
\xymatrixcolsep{3pc}
\xymatrix{
{K[d_X]} \pullback\ar[d] \ar@{>->}[r]^-{k_d} \ar@{}[rd]|{(i)} &{X_1} \ar@/^2pc/[rr]^-{c_X} \ar[r]^-{\la d_X,c_X\ra} \ar[d]_-{d_X} \ar@{}[rd]|{(ii)} &{R_{\mathbb{E}_0}} \ar[d]_-{r_1} \ar[r]^-{r_2} \ar@{}[rd]|{(iii)} &{X_0} \ar@{->>}[d]^-{f_0} \\
0 \ar[r] &{X_0} \pushout \ar@{=}[r] &{X_0} \pushout \ar@{->>}[r]_-{f_0} &{M} \pushout
}
\end{equation*}
the square $(i)$ is both a pullback and a pushout, $(ii)$ is a pushout because $\la d_X,c_X\ra$ is epimorphic, and $(iii)$ is a pushout because $f_0$ is a quotient of $\mathcal{R}_{\mathbb{E}_0}.$ Thus the combined square $(i)+(ii)+(iii)$ is a pushout, and we are done.
\end{proof}
\begin{remark}
\label{rmk:h}
By point $(1)$ and the cancellativity of $K_0,$ for every $x\in X_1$ we have $c_X(x)=k_0(a)+d_X(x)$ for a unique $a\in K_0;$ this defines a map of sets $h\colon X_1\dashrightarrow K_0,$ $h(x)=a.$ Then, since $K_0$ is commutative and $k_0,$ $d_X$ and $c_X$ are monoid homomorphisms, the composite map $hk_1$ is actually a monoid homomorphism, which we shall likewise denote by $h\colon K_1\rightarrow K_0.$
\end{remark}
Observe that, since $f_0\delta=f_0c_Xk_d=f_0d_Xk_d=0,$ by the universal property of the kernel $k_0$ of $f_0$ the crossed semimodule $\delta=c_Xk_d$ factorises as $\delta=k_0g$
\begin{equation*}
\xymatrix{
{K[d_X]} \ar[rr]^-{\delta} \ar@{.>}[rd]_-g &\ &{X_0} \ar@{->>}[r]^-{f_0} &{M} \\
&{K_0} \ar@{>->}[ru]_-{k_0} &\ &\
}
\end{equation*}
for a unique monoid homomorphism $g\colon K[d_X]\rightarrow K_0,$ which is simply the restriction of the monoid homomorphism $h\colon K_1\rightarrow K_0$ of Remark \ref{rmk:h} to the subobject $K[d_X]$ of $K_1.$

We prove the following:
\begin{proposition}
\label{prop:g_is_Schreier}
If $\la d_X,c_X\ra\colon X_1\rightarrow R_{\mathbb{E}_0}$ is a Schreier extension of monoids, so is $g.$
\end{proposition}
\begin{proof}
Immediate by the pullback stability of Schreier extensions, because in the commutative diagram
\begin{equation*}
\xymatrix{
{K[d_X]} \ar@{>>}[d]_-g \ar@{>->}[r]^-{k_d} &{X_1} \ar@{>>}[d]_-{\la d_X,c_X\ra} \ar@{->>}[r]^-{d_X} &{X_0} \ar@{=}[d] \\
{K_0} \ar@{>->}[r]_-{\la0,k_0\ra} &{R_{\mathbb{E}_0}} \ar@{->>}[r]_{r_1} &{X_0,}
}
\end{equation*}
whose rows are short exact sequences in $\mathbf{Mon},$ the left-hand square is a pullback.
\end{proof}
Accordingly, if $\la d_X,c_X\ra\colon X_1\rightarrow R_{\mathbb{E}_0}$ is a Schreier extension, in the diagram
\begin{equation}
\begin{aligned}
\label{eqn:factorisation_crossed_semimodule}
\xymatrix{
{A} \ar@{>->}[r]^-{k} &{K[d_X]} \ar[rr]^-{\delta} \ar@{->>}[rd]_-g &\ &{X_0} \ar@{->>}[r]^-{f_0} &{M} \\
&\ &{K_0} \ar@{>->}[ru]_-{k_0} &\ &\
}
\end{aligned}
\end{equation}
whose upper row is the sequence associated with the crossed semimodule $\delta,$ we have that $\delta=k_0g$ is the (regular epimorphism, monomorphism) factorisation of $\delta,$ and the kernel $(A,k)$ of $\delta$ is also a kernel of $g.$

Moreover, the monoids $K_0$ and $K_1$ are commutative and cancellative by assumption, so that $K[d_X]$ is also commutative and cancellative (it is a submonoid of $K_1,$ by the equality $f_0d_X=f_1$). This in turn implies that $A,$ which is a submonoid of $K[d_X],$ is commutative and cancellative.

Thus, we have proved that, starting with a Schreier internal category~\eqref{eqn:schreier_category_SExt_M_with_Es} in $cc\text{-}SExt_M$ such that the induced morphism $\la d_X,c_X\ra\colon X_1\rightarrow R_{\mathbb{E}_0}$ is a Schreier extension of monoids, the corresponding sequence~\eqref{eqn:factorisation_crossed_semimodule} meets the conditions $\mathbf{CS}_r,$ $\mathbf{CS}_l$ of the previous section, i.e. it is a crossed Schreier extension in the sense of Definition \ref{def:crossed_Schreier_extension}.

Conversely, suppose that we have a crossed semimodule $(B\xlongrightarrow{\delta}C,\alpha),$ with $B$ commutative, such that the associated sequence
\begin{equation}
\label{eqn:crossed_sequence_2}
\begin{aligned}
\xymatrix{
A \ar@{>->}[r]^-k &B \ar@{->>}[rd]_-{\rho} \ar[rr]^-{\delta} & \ &C \ar@{->>}[r]^-{\pi} &{M} \\
 &\ &N \ar@{>->}[ru]_-{\nu}&\ &\ &\
}
\end{aligned}
\end{equation}
is a crossed Schreier extension.

It is straightforward that the commutativity of $B$ implies the commutativity of $A$ and $N,$ and moreover:
\begin{lemma}
\label{lemma:B_cancellative}
Under the above assumptions on~\eqref{eqn:crossed_sequence_2}:
\begin{enumerate}
\item the monoid $B$ is cancellative;
\item the composite action $\alpha\nu\colon N\rightarrow\mathrm{End}(B)$ is the trivial action (i.e., $\alpha\nu(n)=1_B$ for all $n\in N$).
\end{enumerate}
\end{lemma}
\begin{proof}
\begin{enumerate}
\item We have a Schreier extension $\xymatrix{A \ar@{>->}[r]^-k &B \ar@{->>}[r]^-{\rho} &N}$ where both $A$ and $N$ are cancellative. Let $b_1,b_2,b_3\in B$ be such that $b_2+b_1=b_3+b_1.$ Then $\rho(b_2)+\rho(b_1)=\rho(b_3)+\rho(b_1),$ and since $N$ is cancellative it follows that we can write $b_2=k(a_2)+u$ and $b_3=k(a_3)+u,$ for the same representative $u\in B$ of $\rho(b_2)=\rho(b_3)$ and two unique elements $a_2,a_3\in A.$

Observe now that since $A$ is cancellative, for every $b\in B$ and $a,a'\in A$ one has $k(a)+b=k(a')+b$ if and only if $a=a',$ simply by writing $b=k(a'')+v$ for some representative $v$ of $\rho(b)$ and by cancelling $a''.$

Then the equality $b_2+b_1=b_3+b_1,$ which translates into $k(a_2)+(u+b_1)=k(a_3)+(u+b_1),$ entails that $a_2=a_3,$ so that $b_2=b_3.$
\item Write $c\ast b$ for the action $\alpha(c)(b).$ Then, since $\rho$ is surjective, for every $n\in N$ we have $\nu(n)\ast b=\nu\rho(b')\ast b=\delta(b')\ast b$ for some $b'\in B.$

Recall now that, by Definition~\ref{def:crossed_semimodule}$(2),$ there is an equality $\delta(b_1)\ast b_2+b_1=b_1+b_2$ for every $b_1,b_2\in B,$ which means (since $B$ is commutative and cancellative) that $\delta(b_1)\ast b_2=b_2$ for any $b_1,b_2\in B.$

We conclude that indeed $\nu(n)\ast b=b$ for all $n\in N$ and $b\in B.$
\end{enumerate}
\end{proof}
The Schreier internal category corresponding to $\delta$ by Theorem~\ref{thm:equivalence_crossed_semimodules_schreier_categories} is the category $\xymatrix{{B\rtimes_{\alpha} C} \ar@<.9ex>[r]^-{p} \ar@<-.9ex>[r]_-{\lambda} &{C} \ar[l]|-s}$ in $\mathbf{Mon},$ with $p(b,c)=c,$ $s(c)=(0,c)$ and $\lambda(b,c)=\delta(b)+c.$ Then $\pi p=\pi\lambda$ (because $\pi\delta=0$), and we obtain a commutative diagram
\begin{equation}
\label{eqn:associated_category}
\begin{aligned}
\xymatrix{
{B\rtimes_{\alpha}C} \ar@{->>}[rd]_-{\pi_1} \ar@<.9ex>[rr]^-{p} \ar@<-.9ex>[rr]_-{\lambda} &\ &{C} \ar@{->>}[ld]^-{\pi} \ar[ll]|-{s} \\
&{M}
}
\end{aligned}
\end{equation}
where $\pi_1=\pi p=\pi\lambda.$

Observe that, by Lemma~\ref{lemma:B_cancellative}$(2),$ the semidirect product $B\rtimes_{\alpha\nu}N$ is just the direct product $B\times N,$ so that a kernel of $\pi_1$ is given by $1_B\times\nu\colon B\times N\rightarrowtail B\rtimes_{\alpha}C.$
\begin{proposition}
\label{prop:semidirect_is_Schreier}
Under the above assumptions on \eqref{eqn:crossed_sequence_2}:
\begin{enumerate}
\item the sequence
\begin{equation*}
\xymatrix{{{\mathbb{E}}_1:B\times N} \ar@{>->}[r]^-{1_B\times\nu} &{ B\rtimes_{\alpha}C} \ar@{->>}[r]^-{\pi_1} &M}
\end{equation*}
and the diagram~\eqref{eqn:associated_category} belong to the category $cc\text{-}SExt_M;$
\item the induced morphism $\la p,\lambda\ra\colon B\rtimes_{\alpha}C\rightarrow R_{{\mathbb{E}}_0},$ where $\xymatrixcolsep{1.5pc}\xymatrix{{{\mathbb{E}}_0:N} \ar@{>->}[r]^-{\nu} &{C} \ar@{->>}[r]^-{\pi} &{M,}}$ is a Schreier extension.
\end{enumerate}
\end{proposition}
\begin{proof}
\begin{enumerate}
\item For any $(b,c)\in B\rtimes_{\alpha}C$ and any representative $u\in C$ of $\pi(c),$  we have a (clearly unique) decomposition
\begin{equation*}
(b,c)=\big(b,\nu(n)+u\big)=\big(b,\nu(n)\big)+(0,u)=\big(b,\nu(n)\big)+s(u).
\end{equation*}
Thus ${\mathbb{E}}_1$ is a Schreier extension, with representatives $s(u),$ and $s$ preserves the representatives by construction. Since $ps=\lambda s=1_C,$ we conclude that $p$ and $\lambda$ preserve the representatives as well. Finally, since $B$ and $N$ are commutative and cancellative (using Lemma~\ref{lemma:B_cancellative}), the kernel $B\times N$ is also commutative and cancellative.
\item A kernel of $\la p,\lambda\ra$ is given by $\la k,0\ra\colon A\rightarrowtail B\rtimes_{\alpha}C,$ and it is immediate to check that $\xymatrix{{A} \ar@{>->}[r]^-{\la k,0\ra} &{ B\rtimes_{\alpha}C} \ar@{->>}[r]^-{\la p,\lambda\ra} &{R_{{\mathbb{E}}_0}}}$ is a Schreier extension with representatives $(u,c),$ for any $c\in C$ and any representative $u\in B$ for the Schreier extension $\rho.$
\end{enumerate}
\end{proof}
All this motivates the following:
\begin{definition}
\label{def:R-full}
A Schreier internal category \eqref{eqn:schreier_category_SExt_M_with_Es} in $cc\text{-}SExt_M$ is \emph{$\mathcal{R}$-full} if the induced morphism $\la d_X,c_X\ra\colon X_1\longrightarrow R_{\mathbb{E}_0}$ is a Schreier extension of monoids.
\end{definition}
By the results of this section, an $\mathcal{R}$-full Schreier internal category can be equivalently described as a crossed semimodule $(B\xlongrightarrow{\delta}C,\alpha)$ whose domain $B$ is commutative, and such that the corresponding crossed extension
\begin{equation*}
\xymatrix{
A \ar@{>->}[r]^-k &B \ar@{->>}[rd]_-{\rho} \ar[rr]^-{\delta} & \ &C \ar@{->>}[r]^-{\pi} &{M} \\
 &\ &N \ar@{>->}[ru]_-{\nu}&\ &\ &\
}
\end{equation*}
is a crossed Schreier extension in the sense of Definition \ref{def:crossed_Schreier_extension}. In this situation, $A$ and $N$ are necessarily also commutative, and $B$ is necessarily also cancellative (by Lemma \ref{lemma:B_cancellative}).

We now want to show that this correspondence can be made into an actual equivalence of categories, thanks to Theorem \ref{thm:equivalence_crossed_semimodules_schreier_categories}, by introducing appropriate notions of morphisms. Before doing this, we shall comment on how the notion of $\mathcal{R}$-full Schreier internal category relates to the one of  \emph{aspherical abelian groupoid} \cite{aspherical}, in the case of groups.
\begin{remark}
\label{rmk:aspherical_groupoids}
Let $\mathcal{C}$ be a category with finite limits, and denote by $1$ a terminal object in $\mathcal{C}.$ In the terminology of \cite{bourn-direction, aspherical}, one says that an object $X\in\mathcal{C}$ has \emph{global support} if the terminal morphism $X\longrightarrow1$ is a regular epimorphism, and an internal groupoid $\xymatrix{{\underline{X}: X_1\times_{X_0}X_1} \ar[r]^-{m_X} &{X_1} \ar@<.9ex>[r]^-{d_X} \ar@<-.9ex>[r]_-{c_X} &{X_0} \ar[l]|-{e_X}}$ in $\mathcal{C}$ is called \emph{aspherical} when it is \emph{connected} (meaning that the induced morphism $\la d_X,c_X\ra\colon X_1\longrightarrow X_0\times X_0$ is a regular epimorphism) and the object $X_0$ has global support. Moreover, denote by $\mathrm{Mal}(\mathcal{C})$ the category of internal Mal'tsev algebras in $\mathcal{C},$ i.e. pairs $(X,p)$ where $X\in\mathcal{C}$ and $p\colon X\times X\times X\rightarrow X$ is a Mal'tsev operation on $X$ (the morphisms being morphisms in $\mathcal{C}$ which preserve the Mal'tsev operation). If $\mathcal{C}$ is a Mal'tsev category \cite{CLP}, any $X\in\mathcal{C}$ admits at most one Mal'tsev operation, and can therefore be regarded as an object of $\mathrm{Mal}(\mathcal{C})$ in a unique way, provided that such a $p$ exists. Further, if $\mathcal{C}$ is a naturally Mal'tsev category \cite{johnstone}, every $X\in\mathcal{C}$ admits a natural Mal'tsev operation, so that $\mathcal{C}\cong\mathrm{Mal}(\mathcal{C}).$

 Now, in the general case where $\mathcal{C}$ is finitely complete, one can prove that the functor $\mathcal{O}\colon\mathrm{Gpd}(\mathcal{C})\rightarrow\mathcal{C},$ $\underline{X}\mapsto X_0,$ defined by sending a groupoid to its object of objects, is a fibration whose fibres $\mathcal{O}^{-1}(X_0)$ are Mal'tsev categories. Moreover, the fibres $\mathcal{O}^{-1}(X_0)$ are naturally Mal'tsev categories if $\mathcal{C}$ is a Mal'tsev category. In the terminology of \cite{aspherical}, a groupoid $\underline{X}$ is \emph{abelian} if it belongs to $\mathrm{Mal}\big(\mathcal{O}^{-1}(X_0)\big).$ Accordingly, every groupoid in a Mal'tsev category is abelian. (See for example \cite[Section 2.11]{borceux-bourn} for a proof of these facts.)

If $G$ is a group, an object  $f\colon X\longrightarrow G$ in the slice category $\mathcal{C}=\mathbf{Gp}/G$ has global support if and only if it is a surjective group homomorphism, and it is an internal Mal'tsev algebra if and only if its kernel is abelian (cf.~\cite{bourn-direction}). An internal groupoid $\underline{X}_G$ in $\mathcal{C},$ which is automatically abelian, and is the same as an internal category in $\mathcal{C}$ (because $\mathcal{C}$ is a Mal'tsev category, see~\cite{CPP}), is given by a commutative diagram
\begin{equation*}
\xymatrix{
{X_1} \ar[rd]_-{f_1} \ar@<.9ex>[rr]^-{d_X} \ar@<-.9ex>[rr]_-{c_X} &\ &{X_0} \ar[ld]^-{f_0} \ar[ll]|-{e_X} \\
&G
}
\end{equation*}
whose upper row is (the reflexive graph underlying) an internal groupoid in $\mathbf{Gp}.$

Observe that, in this situation, if $f_0$ has global support, the same is true for $f_1,$ because $f_1=f_0d_X$ and $d_X$ is surjective (it is a split epimorphism in $\mathbf{Gp}$). Then the groupoid $\underline{X}_G$ is (abelian and) aspherical precisely when the three morphisms $f_0,$ $f_1$ and $\la d_X,c_X\ra\colon X_1\longrightarrow\mathrm{Eq}(f_0)$ are surjective. Accordingly, $f_0$ and $f_1$ are in this context (trivially) Schreier extensions of $G.$ If moreover $K_1=ker(f_1)$ is abelian, which implies that $K_0=ker(f_0)$ is also abelian (because the split epimorphism $d_X\colon X_1\rightarrow X_0$ induces a split epimorphism ${d^{\star}_X}\colon K_1\rightarrow K_0$ between the kernels), the relation $\mathcal{R}_{\mathbb{E}_0}$ associated with $f_0$ is defined, and it coincides with the kernel pair $\mathcal{E}q(f_0).$

Thus, an $\mathcal{R}$-full Schreier internal category in the full subcategory $AbExt_G\subseteq cc\text{-}SExt_G$ of \emph{group} extensions of $G$ by an abelian kernel, in the sense of Definition \ref{def:R-full}, is the same as an aspherical (abelian) groupoid in $\mathrm{Mal}(\mathbf{Gp}/G).$
\end{remark}
Now, denote by $com\text{-}\mathrm{XSExt}_M$ the category whose objects are the crossed Schreier extensions \eqref{eqn:crossed_sequence_monoids} with fixed codomain $M,$ such that the domain $B$ of the crossed semimodule $\delta$ is commutative, and whose morphisms (following \cite{AMMPS}) are given by commutative diagrams
\begin{equation*}
\xymatrix{
A \ar@{>->}[r]^-k \ar[d]_-{\varphi^\star} &B \ar[r]^-{\delta} \ar[d]_-{\varphi} &C \ar@{->>}[r]^-{\pi} \ar[d]^-{\psi} &M\ar@{=}[d]  \\
{A'} \ar@{>->}[r]_-{k'} &{B'} \ar[r]_-{\delta'} &{C'} \ar@{->>}[r]_-{\pi'} &M
}
\end{equation*}
such that $(\varphi,\psi)$ is a morphism of crossed semimodules, and such that in the diagram
\begin{equation}
\label{eqn:morphism_crossed_Schreier_extensions_factorised}
\begin{aligned}
\xymatrixrowsep{1.3pc}
\xymatrix{
A \ar[dd]_-{\varphi^\star} \ar@{>->}[r]^-k &B \ar[dd]_-{\varphi} \ar@{->>}[rd]_-{\rho} \ar[rr]^-{\delta} & \ &C  \ar[dd]_-{\psi} \ar@{->>}[r]^-{\pi} &{M}  \ar@{=}[dd] \\
 &\ &N \ar[dd]_(.20){\psi^\star} \ar@{>->}[ru]_-{\nu}&\ &\ &\ \\
{A'} \ar@{>->}[r]^-{k'} &{B'} \ar@{->>}[rd]_-{\rho'} \ar[rr]^(.70){\delta'} & \ &{C'} \ar@{->>}[r]^-{\pi'} &{M,} \\
 &\ &{N'} \ar@{>->}[ru]_-{\nu'}&\ &\ &\
}
\end{aligned}
\end{equation}
where $\psi^\star$ is induced by the universal property of the kernel $(N',\nu'),$ both
\begin{equation*}
\begin{aligned}
{\xymatrix{
A \ar@{>->}[r]^-{k} \ar[d]_-{\varphi^\star} &{B} \ar@{->>}[r]^-{\rho} \ar[d]_-{\varphi} &{N} \ar[d]^-{\psi^\star} \\
{A'} \ar@{>->}[r]_-{k'} &{B'} \ar@{->>}[r]_-{\rho'} &{N',}
}}
\
\
{\xymatrix{
N \ar@{>->}[r]^-{\nu} \ar[d]_-{\psi^\star} &{C} \ar@{->>}[r]^-{\pi} \ar[d]_-{\psi} &{M} \ar@{=}[d]\\
{N'} \ar@{>->}[r]_-{\nu'} &{C'} \ar@{->>}[r]_-{\pi'} &{M,}
}}
\end{aligned}
\end{equation*}
are morphisms of $cc$-Schreier extensions.

(A note of warning: from now on, the notation $()^\star$ shall be used tacitly to indicate the morphism induced by the universal property of some appropriate kernel.)

Denote also by $\mathcal{R}\mathrm{SCat}(cc\text{-}SExt_M)$ the category whose objects are the $\mathcal{R}$-full Schreier internal categories
\begin{equation*}
\begin{aligned}
{
\xymatrixcolsep{0.1pc}
\xymatrix{
{{\mathbb{E}}_1}  \ar@<-.99ex>[d]_{d_X} \ar@<.99ex>[d]^{c_X} &{:}\\
{{\mathbb{E}}_0} \ar[u]|{e_X} &{:}
}
}
{\xymatrixcolsep{2.5pc}\xymatrix{
{K_1} \ar@<-.99ex>[d]_{d^\star_X} \ar@<.99ex>[d]^{c^\star_X} \ar@{>->}[r]^-{k_1} &{X_1} \ar@<-.99ex>[d]_{d_X} \ar@<.99ex>[d]^{c_X} \ar@{->>}[r]^-{f_1} &{M} \ar@{=}[d] \\
{K_0} \ar[u]|{e^\star_X} \ar@{>->}[r]_-{k_0} &{X_0} \ar[u]|{e_X} \ar@{->>}[r]_-{f_0} &{M}
}}
\end{aligned}
\end{equation*}
and whose morphisms
\begin{equation}
\label{eqn:morphism_of_R-full}
\begin{aligned}
{\xymatrixcolsep{0.1pc}
\xymatrix{
{\mathcal{E}} \ar[d] &{:}\\
{\mathcal{L}} &{:}
}
}
{\xymatrix{
{{\mathbb{E}}_1} \ar[d]_-{h_1} \ar@<.9ex>[r]^-{d_X} \ar@<-.9ex>[r]_-{c_X} &{{\mathbb{E}}_0} \ar[d]^-{h_0} \ar[l]|{e_X} \\
{{\mathbb{L}}_1} \ar@<.9ex>[r]^-{d_Y} \ar@<-.9ex>[r]_-{c_Y} &{{\mathbb{L}}_0} \ar[l]|{e_Y}
}}
\end{aligned}
\end{equation}
are the internal functors
\begin{equation*}
\xymatrix{
{X_1} \ar[d]_-{h_1} \ar@<.9ex>[r]^-{d_X} \ar@<-.9ex>[r]_-{c_X} &{X_0} \ar[d]^-{h_0} \ar[l]|{e_X} \\
{Y_1} \ar@<.9ex>[r]^-{d_Y} \ar@<-.9ex>[r]_-{c_Y} &{Y_0} \ar[l]|{e_Y}
}
\end{equation*}
in $\mathrm{SCat}(\mathbf{Mon})$ such that
\begin{equation*}
\begin{aligned}
{\xymatrix{
{{\mathbb{E}}_0:K_0} \ar@<1.7ex>[d]_-{h^\star_0} \ar@{>->}[r]^-{k_0} &{X_0} \ar[d]^-{h_0} \ar@{->>}[r]^-{f_0} &M \ar@{=}[d]\\
{{\mathbb{L}}_0:L_0} \ar@{>->}[r]_-{l_0} &{Y_0} \ar@{->>}[r]_-{g_0} &{M,}
}}
\
\
\
{\xymatrix{
{{\mathbb{E}}_1:K_1} \ar@<1.7ex>[d]_-{h^\star_1} \ar@{>->}[r]^-{k_1} &{X_1} \ar[d]^-{h_1} \ar@{->>}[r]^-{f_1} &M \ar@{=}[d]\\
{{\mathbb{L}}_1:L_1} \ar@{>->}[r]_-{l_1} &{Y_1} \ar@{->>}[r]_-{g_1} &{M,}
}}
\end{aligned}
\end{equation*}
and also the diagram
\begin{equation}
\label{eqn:condition_on_R}
\begin{aligned}
\xymatrixcolsep{3pc}
{\xymatrix{
{A} \ar[d]_-{h^\star_1} \ar@{>->}[r]^-{k_X} &{X_1} \ar[d]^-{h_1} \ar@{->>}[r]^-{\la d_X,c_X\ra} &{R_{{\mathbb{E}}_0}} \ar[d]^-{h_0\times h_0}\\
{A'} \ar@{>->}[r]_-{k_Y} &{Y_1} \ar@{->>}[r]_-{\la d_Y,c_Y\ra} &{R_{{\mathbb{L}}_0},}
}}
\end{aligned}
\end{equation}
stemming out of the fact that $\mathcal{E}$ and $\mathcal{L}$ are $\mathcal{R}$-full, are morphisms of Schreier extensions. (Of course, in \eqref{eqn:condition_on_R} $(A,k_X)$ and $(A',k_Y)$ are kernels of $\la d_X,c_X\ra$ and $\la d_Y,c_Y\ra,$ respectively.)

Then, consider the functor
\begin{equation}
\label{eqn:F}
F\colon\mathcal{R}\mathrm{SCat}(cc\text{-}SExt_M)\longrightarrow com\text{-}\mathrm{XSExt}_M
\end{equation}
mapping $\mathcal{E},$ $\mathcal{L}$ and a morphism \eqref{eqn:morphism_of_R-full} to
\begin{equation*}
\xymatrixrowsep{1.8pc}
\xymatrix{
A \ar[dd]_-{h^\star_1} \ar@{>->}[r]^-k &{K[d_X]} \ar[dd]_-{h^\star_1} \ar@{->>}[rd]_-{g} \ar[rr]^-{\delta=c_Xk_d} & \ &{X_0}  \ar[dd]^-{h_0} \ar@{->>}[r]^-{f_0} &{M}  \ar@{=}[dd] \\
 &\ &{K_0} \ar[dd]_(.20){h^\star_0} \ar@{>->}[ru]_-{k_0}&\ &\ &\ \\
{A'} \ar@{>->}[r]^-{k'} &{K[d_Y]} \ar@{->>}[rd]_-{g'} \ar[rr]^(.70){\delta'=c_Yl_d} & \ &{Y_0} \ar@{->>}[r]^-{g_0} &{M,} \\
 &\ &{L_0} \ar@{>->}[ru]_-{l_0}&\ &\ &\
}
\end{equation*}
where $(K[d_X],k_d)$ is a kernel of $d_X,$ $(K[d_Y],l_d)$ is a kernel of $d_Y,$ $(A,k)$ is a kernel of $\delta=c_Xk_d$ and $(A',k')$ is a kernel of $\delta'=c_Yl_d.$

We already know that $F$ is well defined on the objects, and by Theorem \ref{thm:equivalence_crossed_semimodules_schreier_categories} $F(h_1,h_0)$ is indeed a morphism of crossed extensions. Moreover, the diagram
\begin{equation*}
{\xymatrix{
{K_0} \ar[d]_-{h^\star_0} \ar@{>->}[r]^-{k_0} &{X_0} \ar[d]^-{h_0} \ar@{->>}[r]^-{f_0} &M \ar@{=}[d]\\
{L_0} \ar@{>->}[r]_-{l_0} &{Y_0} \ar@{->>}[r]_-{g_0} &{M,}
}}
\end{equation*}
is a morphism of Schreier extensions, by our definition of morphism in the category $\mathcal{R}\mathrm{SCat}(cc\text{-}SExt_M),$ and the fact that the same is true for
\begin{equation*}
{\xymatrix{
{A} \ar[d]_-{h^\star_1} \ar@{>->}[r]^-{k} &{K[d_X]} \ar[d]^-{h^\star_1} \ar@{->>}[r]^-{g} &{K_0} \ar[d]^-{h^\star_0}\\
{A'} \ar@{>->}[r]_-{k'} &{K[d_Y]} \ar@{->>}[r]_-{g'} &{L_0}
}}
\end{equation*}
comes at once from Proposition \ref{prop:g_is_Schreier} and the fact that \eqref{eqn:condition_on_R} is a morphism of Schreier extensions.

Thus, $F$ is well defined also on the arrows, and its functoriality is established by the fact that it is essentially the same as the functor $\mathrm{SCat}(\mathbf{Mon})\rightarrow\mathbf{Xsmod}$ in Theorem \ref{thm:equivalence_crossed_semimodules_schreier_categories}.

Conversely, define
\begin{equation}
\label{eqn:G}
G\colon com\text{-}\mathrm{XSExt}_M\longrightarrow\mathcal{R}\mathrm{SCat}(cc\text{-}SExt_M)
\end{equation}
by sending \eqref{eqn:morphism_crossed_Schreier_extensions_factorised} to
\begin{equation*}
{\xymatrix{
{{\mathbb{E}}_1} \ar[d]_-{\varphi\rtimes\psi} \ar@<.9ex>[r]^-{p} \ar@<-.9ex>[r]_-{\lambda} &{{\mathbb{E}}_0} \ar[d]^-{\psi} \ar[l]|{s} \\
{{\mathbb{L}}_1} \ar@<.9ex>[r]^-{p'} \ar@<-.9ex>[r]_-{\lambda'} &{{\mathbb{L}}_0,} \ar[l]|{s'}
}}
\end{equation*}
where
\begin{itemize}
\item[] $\xymatrix{{{\mathbb{E}}_0:N} \ar@{>->}[r]^-{\nu} &{C} \ar@{->>}[r]^-{\pi} &{M,}}$
\item[] $\xymatrixcolsep{4.5pc}\xymatrix{{{\mathbb{E}}_1:B\times N } \ar@{>->}[r]^-{1_B\times\nu} &{B\rtimes_{\alpha}C} \ar@{->>}[r]^-{\pi_1=\pi p=\pi\lambda} &{M,}}$
\item[] $\xymatrix{{{\mathbb{L}}_0:N'} \ar@{>->}[r]^-{\nu'} &{C'} \ar@{->>}[r]^-{\pi'} &{M,}}$
\item[] $\xymatrixcolsep{4.5pc}\xymatrix{{{\mathbb{L}}_1:B'\times N'} \ar@{>->}[r]^-{1_{B'}\times\nu'} &{B'\rtimes_{\alpha'}C'} \ar@{->>}[r]^-{\pi'_1=\pi' p'=\pi'\lambda'} &{M,}}$
\end{itemize}
and $p,$ $s,$ $\lambda$ are as in \eqref{eqn:associated_category} (and similarly for $p',$ $s'$ and $\lambda'$).

Then we know from Theorem \ref{thm:equivalence_crossed_semimodules_schreier_categories} that if $\alpha\colon C\rightarrow\mathrm{End}(B)$ and $\alpha'\colon C'\rightarrow\mathrm{End}(B')$ are the actions in the crossed semimodules $\delta$ and $\delta',$ the diagram
\begin{equation*}
\xymatrix{
{B\rtimes_{\alpha}C} \ar[d]_-{\varphi\rtimes\psi} \ar@<.9ex>[r]^-{p} \ar@<-.9ex>[r]_-{\lambda} &{C} \ar[d]^-{\psi} \ar[l]|-{s} \\
{B'\rtimes_{\alpha'}C'} \ar@<.9ex>[r]^-{p'} \ar@<-.9ex>[r]_-{\lambda'} &{C'} \ar[l]|-{s'}
}
\end{equation*}
is a morphism of internal Schreier categories in $\mathbf{Mon},$ and by the assumption
\begin{equation*}
\xymatrix{
N \ar@{>->}[r]^-{\nu} \ar[d]_-{\psi^\star} &{C} \ar@{->>}[r]^-{\pi} \ar[d]^-{\psi} &{M} \ar@{=}[d]\\
{N'} \ar@{>->}[r]_-{\nu'} &{C'} \ar@{->>}[r]_-{\pi'} &{M}
}
\end{equation*}
is a morphism of Schreier extensions. Moreover, ${\mathbb{E}}_1$ is a Schreier extension by Proposition \ref{prop:semidirect_is_Schreier}$(1),$ with representatives $s(u)=(0,u)$ for any representative $u\in C$ for ${\mathbb{E}}_0,$ and similarly for ${\mathbb{L}}_1.$ Thus, the fact that $\psi$ preserves the representatives in the diagram above entails that
\begin{equation*}
\xymatrixcolsep{3pc}
\xymatrix{
{B\times N } \ar[d]_-{(\varphi\rtimes\psi)^\star=\varphi\times \psi^\star} \ar@{>->}[r]^-{1_B\times\nu} &{B\rtimes_{\alpha}C}  \ar[d]^-{\varphi\rtimes\psi} \ar@{->>}[r]^-{\pi_1} &{M} \ar@{=}[d] \\
{B'\times N'} \ar@{>->}[r]_-{1_{B'}\times\nu'} &{B'\rtimes_{\alpha'}C'} \ar@{->>}[r]_-{\pi'_1} &{M}
}
\end{equation*}
is also a morphism of Schreier extensions.

Finally, the fact that by assumption
\begin{equation*}
\xymatrix{
A \ar@{>->}[r]^-{k} \ar[d]_-{\varphi^\star} &{B} \ar@{->>}[r]^-{\rho} \ar[d]^-{\varphi} &{N} \ar[d]^-{\psi^\star}\\
{A'} \ar@{>->}[r]_-{k'} &{B'} \ar@{->>}[r]_-{\rho'} &{N'}
}
\end{equation*}
is a morphism of Schreier extensions guarantees that the diagram
\begin{equation*}
\xymatrixcolsep{3pc}
\xymatrix{
A \ar@{>->}[r]^-{\la k,0\ra} \ar[d]_-{\varphi^\star} &{B\rtimes_{\alpha}C} \ar@{->>}[r]^-{\la p,\lambda\ra} \ar[d]^-{\varphi\rtimes\psi} &{R_{{\mathbb{E}}_0}} \ar[d]^-{\psi\times\psi}\\
{A'} \ar@{>->}[r]_-{\la k',0\ra} &{B'\rtimes_{\alpha'}C'} \ar@{->>}[r]_-{\la p',\lambda'\ra}  &{R_{{\mathbb{L}}_0},}
}
\end{equation*}
whose rows are Schreier extensions by Proposition \ref{prop:semidirect_is_Schreier}$(2),$ is also a morphism of Schreier extensions. (Recall that any couple $(u,c),$ such that $c\in C$ and $u\in B$ is a representative for $\rho,$ is a representative for the upper row, and similarly for the lower one.)

Accordingly, as was for $F,$ $G$ is a well defined functor, and by Theorem \ref{thm:equivalence_crossed_semimodules_schreier_categories} we have:
\begin{theorem}
\label{thm:equivalence_RSCat-XSExt}
For any monoid $M,$ the above functors \eqref{eqn:F} and \eqref{eqn:G} yield an equivalence of categories
\begin{equation*}
\mathcal{R}\mathrm{SCat}(cc\text{-}SExt_M)\cong com\text{-}\mathrm{XSExt}_M
\end{equation*}
between the category of $\mathcal{R}$-full Schreier internal categories in $cc\text{-}SExt_M,$ in the sense of Definition \ref{def:R-full} with the morphisms \eqref{eqn:morphism_of_R-full}, and the category of crossed Schreier extensions
\begin{equation*}
\xymatrix{
A \ar@{>->}[r]^-k &B \ar[r]^-{\delta} &C \ar@{->>}[r]^-{\pi} &M
}
\end{equation*}
in the sense of \cite{AMMPS}, with $B$ commutative.
\end{theorem}
\section{Definition of the $1$-dimensional direction functor for $cc\text{-}SExt_M$}
Let $G$ be a group. It is well known that the process of associating with a crossed extension \eqref{eqn:crossed_sequence} the $G$-module occurring from the induced action \eqref{eqn:induced_action_crossed_sequence} is functorial, yielding a functor
\begin{equation}
\label{eqn:action_groups}
\mathrm{XExt}_G\longrightarrow\mathrm{Mod}_G.
\end{equation}
By the results of \cite{AMMPS} recalled in Section \ref{sec:crossed_semimodules}, the same procedure is possible when $G$ is replaced with a monoid $M,$ and the category $\mathrm{XExt}_G$ of crossed extensions ending in $G$ is replaced with the category $\mathrm{XSExt}_M$ of crossed Schreier extensions ending in $M.$ Coherently, the codomain of \eqref{eqn:action_groups} has to be replaced with the category $can\text{-}\mathcal{S}\mathrm{mod}_M$ of cancellative $M$-semimodules (i.e., pairs $(A,\alpha)$ where $A$ is a commutative and cancellative monoid and $\alpha$ is a monoid action of $M$ on $A,$ with action-preserving monoid homomorphisms), and we end up with a functor
\begin{equation}
\label{eqn:action_monoids}
\mathrm{XSExt}_M\longrightarrow can\text{-}\mathcal{S}\mathrm{mod}_M.
\end{equation}
(See \emph{infra}, Proposition \ref{prop:functor_D_1}, for the functoriality of this construction.)

Now, it is shown in \cite{aspherical} that the functor \eqref{eqn:action_groups} enters the bigger picture of a general functor
\begin{equation}
\label{eqn:1-direction-bourn}
d_1\colon\mathrm{AAGpd}(\mathcal{C})\longrightarrow\mathbf{Ab}(\mathcal{C}),
\end{equation}
defined for a Barr-exact category $\mathcal{C},$ which is a $1$-dimensional development of the ($0$-dimensional) direction functor
\begin{equation}
\label{eqn:0-direction-bourn}
d=d_0\colon\mathrm{AutMal}({\mathcal{C}}_g)\longrightarrow\mathbf{Ab}(\mathcal{C})
\end{equation}
introduced in \cite{bourn-direction}. (Here $\mathrm{AutMal}({\mathcal{C}}_g)\subseteq\mathrm{Mal}({\mathcal{C}}_g)$ denotes the full subcategory of the Mal'tsev algebras $(X,p)$ whose Mal'tsev operation $p$ is \emph{autonomous}, meaning that it is itself a morphism of Mal'tsev algebras: see \cite{bourn-direction} for further details.)

The category appearing as the domain of $d_1$ is the category of aspherical abelian ($1$-)groupoids in the full subcategory $\mathcal{C}_g$ of objects with global support in a Barr-exact category $\mathcal{C},$ in the sense recalled in Remark \ref{rmk:aspherical_groupoids}.

The denomination $1$\emph{-dimensional} comes from the fact that, when $\mathcal{C}$ is a naturally Mal'tsev category, $d_1$ corresponds to the case $n=1$ of a family of direction functors $\{d_n\}_{n\in\mathbb{N}},$ where each $d_n$ is defined on the category on aspherical (abelian) $n$-groupoids in $\mathcal{C}$ (see \cite{Bourn-Rodelo}).

In this section, we shall prove that the functor \eqref{eqn:action_monoids} can also be framed in terms of a suitable $1$-dimensional version of the ($0$-dimensional) direction functor
\begin{equation}
\label{eqn:direction_functor}
d=d_0\colon cc\text{-}SExt_M\longrightarrow\mathbf{CMon}(cc\text{-}SExt_M)
\end{equation}
introduced in \cite{AMR}, and which was seen to provide a generalisation to $cc\text{-}SExt_M$ of the classical direction functor \cite{bourn-direction} applied to the slice category $\mathbf{Gp}/G.$ This new functor will similarly provide a generalisation of the functor \eqref{eqn:1-direction-bourn}, when the latter is applied to $\mathcal{C}=\mathbf{Gp}/G.$
\subsection{The $0$-dimensional direction functor for $cc\text{-}SExt_M$}
We begin by recalling the definition and main properties of \eqref{eqn:direction_functor}.

For any object $\xymatrix{{\mathbb{E}:K} \ar@{>->}[r]^-{k} &X \ar@{->>}[r]^-{f} &M}$ in $cc\text{-}SExt_M,$ one considers the \emph{Chasles relation} $\Ch{E}$ on $R_{\mathbb{E}},$ defined by:
\begin{equation*}
\big(x,k(a)+x\big)\Ch{E}\big(y,k(b)+y\big) \Longleftrightarrow (x,y)\in R_{\mathbb{E}} \ \text{ and } \ a=b.
\end{equation*}
It is a reflexive and transitive relation in $\mathbf{Mon}$ (and it is symmetric if and only if $\rel{E}$ is), represented by the graph
\begin{equation*}
\xymatrixcolsep{4pc}
\xymatrix{\Ch{E}: \P{E} \ar@<.9ex>[r]^-{\pi_1=\la r_1p_1,p\ra} \ar@<-.9ex>[r]_-{p_2} &{\R{E}} \ar[l]|-{\sigma_0}}
\end{equation*}
where $P_{\mathbb{E}}$ is a pullback
\begin{equation*}
\xymatrix{
{\P{E}} \pullback \ar[r]^-{p_2} \ar[d]_-{p_1}  &{\R{E}} \ar[d]^-{r_1} \\
{\R{E}} \ar[r]_{r_2}  &{X,}
}
\end{equation*}
$\sigma_0$ is the morphism defined by $\sigma_0\big(x,k(a)+x\big)=\big(x,x,k(a)+x\big),$ and moreover $p:P_{\mathbb{E}}\longrightarrow X,$ $p\big(x,k(a)+x,k(a)+k(b)+x\big)=k(b)+x,$ is the connector arising from the fact that the relation $\rel{E}$ is self-centralising (see \cite[Proposition 5.7]{AMR}).

Denoting by $\xymatrixcolsep{1.5pc}\xymatrix{{\gamma\colon R_{\mathbb{E}}} \ar@{->>}[r] &{df}}$ a quotient of $\Ch{E},$ one gets by the universal property of the coequalisers $\gamma$ and $f$ a commutative diagram
\begin{equation*}
\xymatrixcolsep{3pc}
\xymatrixrowsep{3pc}
\xymatrix{
{\P{E}} \ar@<-.9ex>[d]_-{p_1} \ar@<.9ex>[d]^-{\pi_2} \ar@<.9ex>[r]^-{\pi_1} \ar@<-.9ex>[r]_-{p_2} &{\R{E}} \ar[l]|{\sigma_0} \ar@<-.9ex>[d]_-{r_1} \ar@<.9ex>[d]^-{r_2} \ar@{->>}[r]^-{\gamma} &{df} \ar@<-.5ex>[d]_-{\overline{f}}\\
{\R{E}} \ar[u]|{\sigma_1} \ar@<.9ex>[r]^-{r_1} \ar@<-.9ex>[r]_-{r_2} &{X} \ar[l]|{s_0} \ar[u]|{s_0} \ar@{->>}[r]_-f &M \ar@<-.5ex>[u]_-{\overline{s}}
}
\end{equation*}
where $\overline{f}\colon\gamma\big(x,k(a)+x\big)\mapsto f(x)$ and $\overline{s}\colon m=f(x)\mapsto\gamma(x,x).$ The morphisms $\pi_2$ and $\sigma_1$ in the left-hand side are given by $\pi_2=\la p,r_2p_2\ra$ and $\sigma_1\colon\big(x,k(a)+x\big)\mapsto\big(x,k(a)+x,k(a)+x\big),$ respectively.

One can prove, then, that the couple $(\overline{f},\overline{s})$ is a Schreier point on $M$ (\cite[Proposition 6.3]{AMR}), and thanks to the equivalence $cc\text{-}SPt_M\cong\mathbf{CMon}(cc\text{-}SExt_M)$ between the categories of Schreier points on $M$ having a commutative and cancellative kernel and of internal commutative monoids in $cc\text{-}SExt_M$ (\cite[Theorem 6.5]{AMR}), the point $(\overline{f},\overline{s})$ may be also described as an internal commutative monoid $\overline{f}$ in $cc\text{-}SExt_M,$ with neutral element $\overline{s}$ and monoid operation
\begin{equation}
\label{eqn:monoid_operation_on_direction}
\begin{aligned}
\begin{split}
\gamma\big(x,k(a)+x\big)\cdot\gamma\big(y,k(b)+y\big)&=\gamma\big(x,k(a)+k(b)+x\big)\\
&=\gamma\big(y,k(a)+k(b)+y\big),
\end{split}
\end{aligned}
\end{equation}
defined whenever $(x,y)\in\mathrm{Eq}(f)$ (see \cite[Corollary 6.8]{AMR}).

Thus, we define on the objects $d(\mathbb{E})=(\overline{f},\overline{s}),$ and by resorting to the universal property of the coequaliser to extend the definition of $d$ to the arrows, one obtains a functor \eqref{eqn:direction_functor}, which is called the \emph{direction functor for} $cc$\emph{-Schreier extensions} (because it extends the classical direction functor \cite{bourn-direction} applied to the slice category $\mathbf{Gp}/G,$ as already mentioned; see \cite{AMR} for details).

The functor $d$ admits an equivalent description in terms of actions. Indeed, recall that any $cc$-Schreier extension $\xymatrix{{\mathbb{E}:K} \ar@{>->}[r]^-{k} &X \ar@{->>}[r]^-{f} &M}$ determines a monoid action of $M$ on $K$:
\begin{equation}
\label{eqn:action_induced_by_Schreier_ext}
\eta\colon M\times K\dashrightarrow K, \ (m,a)\mapsto m\ast a \ \text{ with } u_m+k(a)=k(m\ast a)+u_m,
\end{equation}
where $u_m$ is any representative of $m$ (see for example \cite[Remark 4.20]{P-II} or \cite[Section 3]{AMR}). At the same time, there is an equivalence between Schreier points on $M$ and monoid actions of $M,$ by which a Schreier point $\xymatrix{B \ar@<.5ex>[r]^-g &M \ar@<.5ex>[l]^-s}$ corresponds to the action
\begin{equation}
\label{eqn:action_induced_by_Schreier_point}
M\times K[g]\dashrightarrow K[g], \ (m,a)\mapsto m\bullet a \ \text{ with } s(m)+k(a)=k(m\bullet a)+s(m),
\end{equation}
where $(K[g],k)$ is a kernel of $g$ (see \cite{schreier_book}, Proposition 5.2.2).

In the case of $d(\mathbb{E})=(\overline{f},\overline{s}),$ the two actions \eqref{eqn:action_induced_by_Schreier_ext}, \eqref{eqn:action_induced_by_Schreier_point} coincide (\cite[Proposition 6.10]{AMR}), and the functor $d$ can be described as the action functor
\begin{equation}
\label{eqn:d_action_guise}
d\colon cc\text{-}SExt_M\longrightarrow can\text{-}\mathcal{S}\mathrm{mod}_M, \  \mathbb{E}\longmapsto (K,\eta).
\end{equation}
Finally, it is proven in \cite[Section 7]{AMR} that $d$ is a conservative cofibration, which preserves finite products, and preserves and reflects monomorphisms and regular epimorphisms. As a consequence, the fibres of $d$ are endowed with a natural symmetric monoidal structure (\cite[Corollary 7.9]{AMR}), and it follows that the commutative monoid obtained from the equivalence classes of connected components of the fibres $d^{-1}(K,\eta)$ are precisely Patchkoria's second cohomology monoids $H^2(M,K,\eta),$ as described in \cite{P-I, P-II}.
\subsection{The $1$-dimensional direction functor for $cc\text{-}SExt_M$}
Now, to understand the idea underlying the construction that follows, recall from \cite{Bourn-Rodelo} that the functor $d_1$ in \eqref{eqn:1-direction-bourn} is defined from the functor $d=d_0$ in \eqref{eqn:0-direction-bourn} as follows. If $\mathcal{C}$ is a naturally Mal'tsev and Barr-exact category, then the image $d_1(\underline{X})$ of an aspherical (abelian) groupoid $\xymatrix{{\underline{X}: X_1} \ar@<.9ex>[r]^-{d_X} \ar@<-.9ex>[r]_-{c_X} &{X_0} \ar[l]|-{e_X}}$ under the functor \eqref{eqn:1-direction-bourn} (namely, its \emph{direction}) is given by a kernel $A=Ker\big(d(\la d_X,c_X\ra)\big)$ in $\mathbf{Ab}(\mathcal{C})$ of the regular epimorphism $\xymatrixcolsep{1pc}\xymatrix{{d(\la d_X,c_X\ra)\colon d(X_1)} \ar@{->>}[r] &{d(X_0)\times d(X_0)}}$ obtained by applying $d$ to the induced morphism $\xymatrix{\la d_X,c_X\ra\colon X_1 \ar@{->>}[r] &{X_0\times X_0.} }$

We know from Remark \ref{rmk:aspherical_groupoids} that, in the case where $\mathcal{C}=\mathrm{Mal}(\mathbf{Gp}/G)$ (so that $\mathcal{C}_g\cong AbExt_G$), such an aspherical groupoid $\underline{X}$ is precisely an $\mathcal{R}$-full Schreier internal category in $AbExt_G,$ in the sense of Definition \ref{def:R-full}: thus, our first step in the attempt to define a $1$-dimensional analogue $d_1$ for the functor \eqref{eqn:direction_functor} will be to make precise the idea of applying $d$ to the Schreier extension $\xymatrixcolsep{1.5pc}\xymatrix{{\la d_X,c_X\ra\colon X_1}\ar@{->>}[r] &{R_{ {\mathbb{E}}_0 }} }$ associated to an $\mathcal{R}$-full Schreier internal category \eqref{eqn:schreier_category_SExt_M_with_Es} in $cc\text{-}SExt_M.$

Fix an object $\xymatrix{{\mathbb{E}:K} \ar@{>->}[r]^-{k} &{X} \ar@{->>}[r]^-{f} &M}$ in $cc\text{-}SExt_M.$ By \cite[Proposition 5.6]{AMR}, the relation $\xymatrix{ {\mathcal{R}_{\mathbb{E}}:R_{\mathbb{E}}}\ar@<.9ex>[r]^-{r_1} \ar@<-.9ex>[r]_-{r_2} &{X} \ar[l]|-{s_0}}$ determines another monoid extension
\begin{equation}
\label{eqn:R-hat}
\xymatrixcolsep{4pc}
\xymatrix{{\widehat{R}_{\mathbb{E}}:K\times K} \ar@{>->}[r]^-{\widehat{k}} &{R_{\mathbb{E}}} \ar@{->>}[r]^-{\widehat{f}=fr_1=fr_2} &M}
\end{equation}
on $M,$ with $\widehat{k}(a,b)=\big(k(a),k(b)+k(a)\big),$ which is a Schreier extensions with representatives $(u_m,u_m)$ (where $u_m$ is a representative of $\mathbb{E}$).

A few considerations are in order, about this Schreier extension.
\begin{remarks}
\label{rmks:R-hat}
\begin{enumerate}
\item If $\mathbb{E}$ is a Schreier point on $M,$ with $f$ split by $t\colon M\longrightarrow X,$ then in \eqref{eqn:R-hat} $\widehat{f}=fr_1=fr_2$ is split by $s_0t\colon M\longrightarrow R_{\mathbb{E}},$ so that $\widehat{R}_{\mathbb{E}}$ is also a Schreier point.
\item Every pair $(u_m,v_m),$ with $u_m$ and $v_m$ (possibly different) representatives for $m$ in $\mathbb{E},$ is actually a representative for $m$ in $\widehat{R}_{\mathbb{E}},$ because in this case $(u_m,v_m)=\widehat{k}(0,a)+(u_m,u_m)$ for a unique invertible element $a\in K.$
\item We already observed that if $(\mathrm{Eq}(f),\rho_1,\rho_2)$ is a kernel pair of $f$ in $\mathbf{Mon},$ the monoid $R_{\mathbb{E}}$ is a subobject $\xymatrixcolsep{1pc}\xymatrix{{j\colon R_{\mathbb{E}}} \ar@{>->}[r] &{\mathrm{Eq}(f).}}$ Then the Schreier extension \eqref{eqn:R-hat}, which responds to the idea of looking at the monoid $R_{\mathbb{E}}$ as an object of $cc\text{-}SExt_M,$ allows to expand $j$ into a monomorphism
\begin{equation*}
\begin{aligned}
{
\xymatrixcolsep{0.1pc}
\xymatrix{
{\widehat{R}_{\mathbb{E}}} &{:}\\
{\mathbb{E}\times\mathbb{E}} &{:}
}
}
{
\xymatrixcolsep{3pc}
\xymatrix{
{K\times K} \ar@{>->}[r]^-{\widehat{k}} \ar@{>->}[d]_-{\la1_K,+\ra} &{R_{\mathbb{E}}} \ar@{>->}[d]^-j \ar@{->>}[r]^-{\widehat{f}} &M \ar@{=}[d] \\
{K\times K} \ar@{>->}[r]_-{k\times_Mk} &{\mathrm{Eq}(f)} \ar@{->>}[r]_{f\rho_1=f\rho_2} &{M}
}
}
\end{aligned}
\end{equation*}
in $cc\text{-}SExt_M,$ which realises $\widehat{R}_{\mathbb{E}}$ as a subobject of $\mathbb{E}\times\mathbb{E}.$ (Observe that $\la1_K,+\ra,$ which is indeed an injective monoid homomorphism when $K$ is commutative and cancellative, is surjective if and only if $K$ is a group: thus $\widehat{R}_{\mathbb{E}}=\mathbb{E}\times\mathbb{E}$ if and only if $K$ is a group, if and only if $\rel{E}$ is symmetric.)

Now, a subobject of an object $(A,\eta)$ in $can\text{-}\mathcal{S}\mathrm{mod}_M$ is a submonoid $B\subseteq A$ which is closed with respect to the action $\eta,$ together with the restricted action: thus, since the direction functor \eqref{eqn:d_action_guise} preserves monomorphisms and the kernels of $\widehat{f}$ and $f\rho_1=f\rho_2$ are the same, it follows that $d(\widehat{R}_{\mathbb{E}})=d(\mathbb{E}\times\mathbb{E}).$ Since moreover $d$ preserves finite products, we see that $d(\widehat{R}_{\mathbb{E}})=d(\mathbb{E}\times\mathbb{E})\cong d(\mathbb{E})\times d(\mathbb{E}),$ and we conclude that the $M$-action on $K\times K$ determined by \eqref{eqn:R-hat} is component-wise the action induced by $\mathbb{E}.$
\end{enumerate}
\end{remarks}
Then, if $\xymatrix{{\mathbb{E}_0:K_0} \ar@{>->}[r]^-{k_0} &{X_0} \ar@{->>}[r]^-{f_0} &{M}}$ and $\xymatrix{{\mathbb{E}_1:K_1} \ar@{>->}[r]^-{k_1} &{X_1} \ar@{->>}[r]^-{f_1} &{M}}$ are $cc$-Schreier extensions on $M$ and $\xymatrix{
{\mathcal{E}:{{\mathbb{E}}_1}\times_{{\mathbb{E}}_0}{{\mathbb{E}}_1}} \ar[r]^-{m_X} &{{\mathbb{E}}_1} \ar@<.9ex>[r]^-{d_X} \ar@<-.9ex>[r]_-{c_X} &{{\mathbb{E}}_0} \ar[l]|-{e_X}
}$ is a Schreier internal category in $cc\text{-}SExt_M,$ since by assumption $d_X$ and $c_X$ are morphisms of Schreier extensions, it follows that the induced morphism $\la d_X,c_X\ra\colon X_1\longrightarrow R_{{\mathbb{E}}_0}$ can be regarded as a morphism
\begin{equation}
\label{eqn:aspherical_in_cc-SExt_M}
\begin{aligned}
{
\xymatrixcolsep{0.1pc}
\xymatrix{
{{\mathbb{E}}_1} \ar[d]_-{\la d_X,c_X \ra} &{:}\\
{\widehat{R}_{{\mathbb{E}}_0}} &{:}
}
}
{
\xymatrixcolsep{3pc}
\xymatrix{
{K_1} \ar[d]_-{\la d^\star_X,h\ra} \ar@{>->}[r]^-{k_1} &{X_1} \ar[d]_-{\la d_X,c_X\ra} \ar@{->>}[r]^-{f_1} &M \ar@{=}[d] \\
{K_0\times K_0} \ar@{>->}[r]_-{\widehat{k_0}} &{R_{{\mathbb{E}}_0}} \ar@{->>}[r]_-{\widehat{f_0}} &{M}
}
}
\end{aligned}
\end{equation}
in $cc\text{-}SExt_M.$ In this diagram, as in Section 3, $d^\star_X\colon K_1\longrightarrow K_0$ denotes the restriction of $d_X\colon X_1\longrightarrow X_0$ to the kernels, and $h\colon K_1\longrightarrow K_0$ denotes the monoid homomorphism introduced in Remark \ref{rmk:h}, realising the equality $c_X\big(k_1(a)\big)=k_0\big(h(a)\big)+d_X\big(k_1(a)\big)$ for every $a\in K_1.$

Next, observe that, for every $\mathbb{E}\in cc\text{-}SExt_M,$ the relation $\rel{E}$ yields morphisms of Schreier extensions
\begin{equation*}
\begin{aligned}
{
\xymatrixcolsep{0.1pc}
\xymatrix{
{\widehat{R}_{{\mathbb{E}}}} \ar@<.99ex>[d]^-{r_2} \ar@<-.99ex>[d]_-{r_1} &{:} \\
{\mathbb{E}} \ar[u]|-{s_0} &{:}
}
}
{
\xymatrixcolsep{3pc}
\xymatrix{
{K\times K} \ar@<.99ex>[d]^-{+} \ar@<-.99ex>[d]_-{p_1} \ar@{>->}[r]^-{\widehat{k}} &{R_{{\mathbb{E}}}}  \ar@<.99ex>[d]^-{r_2} \ar@<-.99ex>[d]_-{r_1} \ar@{->>}[r]^-{\widehat{f}} &{M} \ar@{=}[d] \\
{K} \ar[u]|-{t_0} \ar@{>->}[r]_-k &{X}\ar[u]|-{s_0}  \ar@{->>}[r]_-{f} &{M,}
}
}
\end{aligned}
\end{equation*}
with $p_1=r^\star_1$ the projection on the first factor and $t_0=s^\star_0=\la 1_K,0\ra,$ and since $\rel{E}$ is a reflexive and transitive relation and $(r_1,s_0)$ is a Schreier point (by \cite[Proposition 5.5]{AMR}) we conclude that
\begin{equation}
\label{eqn:R-eq}
\xymatrix{{\mathcal{R}(\mathbb{E}):\widehat{R}_{{\mathbb{E}}}} \ar@<.99ex>[r]^-{r_1} \ar@<-.99ex>[r]_-{r_2} &{\mathbb{E}} \ar[l]|-{s_0}}
\end{equation}
is a Schreier internal category in $cc\text{-}SExt_M;$ furthermore, it is trivially $\mathcal{R}$-full, since the induced morphism $\la r_1,r_2\ra$ is just the identity on $R_{\mathbb{E}}.$

This gives us a functor
\begin{equation}
\label{eqn:R}
\mathcal{R}\colon cc\text{-}SExt_M\longrightarrow\mathcal{R}\mathrm{SCat}(cc\text{-}SExt_M), \ \mathbb{E} \mapsto \mathcal{R}(\mathbb{E}),
\end{equation}
defined on the morphisms by sending
\begin{equation*}
\xymatrix{
{\mathbb{E}:K} \ar@<1.7ex>[d]_-{\varphi^\star} \ar@{>->}[r]^-k &X \ar[d]^-{\varphi} \ar@{->>}[r]^-f &M \ar@{=}[d]\\
{\mathbb{E}^\prime:K^\prime} \ar@{>->}[r]_-{k^\prime} &{X^\prime} \ar@{->>}[r]_-{f^\prime} &{M}
}
\end{equation*}
to
\begin{equation*}
\xymatrix{
{R_{\mathbb{E}}}\ar@<.9ex>[r]^-{r_1} \ar@<-.9ex>[r]_-{r_2} \ar[d]_-{R(\varphi)} &{X} \ar[d]^-{\varphi} \ar[l]|-{s_0} \\
{R_{\mathbb{E}'}}\ar@<.9ex>[r]^-{{r_1}^\prime} \ar@<-.9ex>[r]_-{{r_2}^\prime} &{X^\prime,} \ar[l]|-{{s_0}^\prime}
}
\end{equation*}
where $R(\varphi)=\varphi\times\varphi\colon R_{\mathbb{E}}\longrightarrow R_{\mathbb{E}'},$ $\big(x,k(a)+x\big)\longmapsto\big(\varphi(x),k'\varphi^\star(a)+\varphi(x)\big).$ Similarly, we define a functor
\[\mathcal{R}_S\colon cc\text{-}SExt_M\longrightarrow\mathrm{SCat}(cc\text{-}SExt_M), \ \mathbb{E} \mapsto \mathcal{R}_S(\mathbb{E})=\mathcal{R}(\mathbb{E}),
\]
whose codomain is the category of (non necessarily $\mathcal{R}$-full) Schreier internal categories in $cc\text{-}SExt_M,$ obtained by composition of \eqref{eqn:R} with the forgetful functor
\begin{equation*}
U\colon\mathcal{R}\mathrm{SCat}(cc\text{-}SExt_M)\longrightarrow\mathrm{SCat}(cc\text{-}SExt_M)
\end{equation*}
(which acts as the identity but ``forgets'' $\mathcal{R}$-fullness).

Moreover, by denoting with
\begin{equation}
\label{eqn:O_R}
\mathcal{O}\colon\mathcal{R}\mathrm{SCat}(cc\text{-}SExt_M)\longrightarrow cc\text{-}SExt_M
\end{equation}
and
\begin{equation}
\label{eqn:O_S}
\mathcal{O}_S\colon\mathrm{SCat}(cc\text{-}SExt_M)\longrightarrow cc\text{-}SExt_M
\end{equation}
the ``object of objects'' functors defined by
\begin{equation*}
\xymatrix{ {(\mathcal{E}:{\mathbb{E}}_1} \ar@<.9ex>[r]^-{d_X} \ar@<-.9ex>[r]_-{c_X} &{{\mathbb{E}}_0)} \ar[l]|-{e_X}
}\longmapsto {\mathbb{E}}_0,
\end{equation*}
the fact that \eqref{eqn:aspherical_in_cc-SExt_M} is a morphism of Schreier extensions shows that for every ${\mathbb{E}}_0\in cc\text{-}SExt_M$ the internal category $\mathcal{R}({\mathbb{E}}_0),$ defined as in \eqref{eqn:R-eq}, is a terminal object both in the fibre $\mathcal{O}_S^{-1}({\mathbb{E}}_0)$ of the functor \eqref{eqn:O_S} and in the fibre $\mathcal{O}^{-1}({\mathbb{E}}_0)$ of the functor \eqref{eqn:O_R}: indeed, for any $\xymatrix{ {\mathcal{E}:{\mathbb{E}}_1} \ar@<.9ex>[r]^-{d_X} \ar@<-.9ex>[r]_-{c_X} &{{\mathbb{E}}_0} \ar[l]|-{e_X}
}$ in $(\mathcal{R})\mathrm{SCat}(cc\text{-}SExt_M),$ \eqref{eqn:aspherical_in_cc-SExt_M} provides exactly the terminal morphism
\begin{equation*}
\xymatrixrowsep{2.5pc}
\xymatrix{
{{\mathbb{E}}_1} \ar@<.9ex>[rd]^-{d_X} \ar@<-.9ex>[rd]_-{c_X} \ar[rr]^-{\la d_X,c_X\ra} &\ &{\widehat{R}_{{\mathbb{E}}_0}} \ar@<.9ex>[dl]^-{r_2} \ar@<-.9ex>[dl]_-{r_1} \\
&{{\mathbb{E}}_0.} \ar[lu]|-{e_X} \ar[ru]|-{s_0} &\
}
\end{equation*}
We can derive two consequences from these facts.

First, it is well known that if $\mathcal{C}$ is a finitely complete category then the \emph{object of objects} functor
\begin{equation}
\label{eqn:O_general}
\mathcal{O}\colon\mathrm{Cat}(\mathcal{C})\longrightarrow\mathcal{C}, \ \xymatrix{{\big(X_1} \ar@<.9ex>[r]^-{d_X} \ar@<-.9ex>[r]_-{c_X} &{X_0\big)} \ar[l]|-{e_X} }\longmapsto X_0,
\end{equation}
admits a right adjoint $\nabla$ defined by sending an object $X\in\mathcal{C}$ to the internal category $\xymatrix{{X\times X} \ar@<.9ex>[r]^-{p_1} \ar@<-.9ex>[r]_-{p_2} &X \ar[l]|-{\Delta_X}}$ given by the total relation on $X.$

In the case of $\mathcal{C}=cc\text{-}SExt_M,$ which is not finitely complete because it does not have all pullbacks (cf. \cite[Appendix]{AMR}), it is still true that every object $\mathbb{E}$ has a binary product $\mathbb{E}\times\mathbb{E},$ but the resulting internal category $\nabla(\mathbb{E})$ is not a Schreier internal category, in general, so that the above adjunction is lost when $\mathcal{O}$ is considered on $\mathrm{SCat}(\mathcal{C}).$ The property is restored by replacing $\nabla$ with $\mathcal{R}$:
\begin{proposition}
There is a commutative diagram of categories and functors, forming adjunctions
\begin{equation*}
\begin{aligned}
\xymatrix{
&{\mathrm{SCat}(cc\text{-}SExt_M)} \ar@<-.99ex>[d]_-{\mathcal{O}_S} \ar@{}[d]|-{\dashv} \\
{\mathcal{R}\mathrm{SCat}(cc\text{-}SExt_M)} \ar@/^1.5pc/[ru]^-{U} \ar@<.99ex>[r]^-{\mathcal{O}}  &{cc\text{-}SExt_M.} \ar@<.99ex>[l]^-{\mathcal{R}} \ar@<-.99ex>[u]_-{\mathcal{R}_S}   \ar@{}[l]|-{\bot}
}
\end{aligned}
\end{equation*}
\end{proposition}
\begin{proof}
Both adjunctions are defined by sending a morphism $\varphi\colon {\mathbb{E}}_0\longrightarrow {\mathbb{L}}_0,$ whose domain is the object of objects of $\xymatrix{ {\mathcal{E}:{\mathbb{E}}_1} \ar@<.9ex>[r]^-{d_X} \ar@<-.9ex>[r]_-{c_X} &{{\mathbb{E}}_0,} \ar[l]|-{e_X}
}$ to the morphism
\begin{equation*}
\xymatrix{ {\mathcal{E}} \ar[r]^-{!_{{\mathbb{E}}_0}} &{\mathcal{R}({\mathbb{E}}_0)} \ar[r]^-{R(\varphi)} &{\mathcal{R}({\mathbb{L}}_0)} }
\end{equation*}
where $!_{{\mathbb{E}}_0}$ is the terminal morphism in the appropriate fibre $\mathcal{O}_S^{-1}({\mathbb{E}}_0)$ or $\mathcal{O}^{-1}({\mathbb{E}}_0).$
\end{proof}
The second aspect is that we are now in the position to examine more closely the difference between the notion of $\mathcal{R}$-fullness and the one of (Schreier) aspherical internal category in $cc\text{-}SExt_M,$ as it would come from a direct translation from \cite{aspherical}.

Recall indeed (as we have already mentioned in Remark \ref{rmk:aspherical_groupoids} for the case of groupoids), that the fact that an internal category $\xymatrix{{\underline{X}:X_1} \ar@<.9ex>[r]^-{d_X} \ar@<-.9ex>[r]_-{c_X} &{X_0} \ar[l]|-{e_X}}$ (in a category $\mathcal{C}$ with finite products) is \emph{connected}, i.e. that $\la d_X,c_X\ra\colon X_1\longrightarrow X_0\times X_0$ is a regular epimorphism, can be expressed by saying that $\underline{X}$ has global support in the fibre $\mathcal{O}^{-1}(X_0)$ of the functor \eqref{eqn:O_general}, i.e. by looking at $\la d_X,c_X\ra$ as a morphism from $\underline{X}$ into the terminal object $\xymatrix{{\nabla{X_0}:X_0\times X_0} \ar@<.9ex>[r]^-{p_1} \ar@<-.9ex>[r]_-{p_2} &{X_0} \ar[l]|-{\Delta_{X_0}}}$ of $\mathcal{O}^{-1}(X_0).$

In the case of Schreier internal categories in $\mathcal{C}= cc\text{-}SExt_M,$ it is easy to see that if in \eqref{eqn:aspherical_in_cc-SExt_M} the monoid homomorphism $\la d_X,c_X\ra\colon X_1\longrightarrow R_{{\mathbb{E}}_0}$ is surjective, then $\la d^\star_X,h\ra\colon K_1\rightarrow K_0\times K_0$ is also surjective (the left square in \eqref{eqn:aspherical_in_cc-SExt_M} being a pullback), and it follows by \cite[Proposition A.4]{AMR} that in this situation the morphism $\la d_X,c_X\ra\colon{\mathbb{E}}_1\rightarrow\widehat{R}_{{\mathbb{E}}_0}$ is a regular epimorphism in $cc\text{-}SExt_M.$ We maintain that this is the ``correct'' notion of \emph{connectedness} for an object in $\mathrm{SCat}(cc\text{-}SExt_M),$ since it follows by Proposition \ref{prop:d_X,c_X-inside_R} that $\la d_X,c_X\ra$ can never be a regular epimorphism as a morphism with codomain ${\mathbb{E}}_0\times{\mathbb{E}}_0,$ unless $K_0$ is a group (in which case ${\mathbb{E}}_0\times{\mathbb{E}}_0=\widehat{R}_{{\mathbb{E}}_0}$). Thus, if $\mathcal{E}$ is $\mathcal{R}$-full (so that $\la d_X,c_X\ra\colon X_1\longrightarrow R_{{\mathbb{E}}_0}$ is, moreover, a Schreier extension), $\mathcal{E}$ is also connected, in this sense.

But further, back to the case of a general $\mathcal{C},$ recall that an internal category $\underline{X}$ is \emph{aspherical}, in the sense of \cite{aspherical}, if it is connected and its object of objects $X_0$ has global support in $\mathcal{C}$: for $\mathcal{C}= cc\text{-}SExt_M,$ this condition is trivially satisfied by any (Schreier) internal category $\mathcal{E},$ since for every ${\mathbb{E}}_0$ the terminal morphism
\begin{equation*}
\xymatrix{
{{\mathbb{E}}_0:K_0} \ar@<1.7ex>@{->>}[d] \ar@{>->}[r]^-{k_0} &{X_0} \ar@{->>}[d]^-{f_0} \ar@{->>}[r]^-{f_0} &M \ar@{=}[d]\\
{{\mathds{1}}_M:0 \hspace{0.3cm} } \ar@{>->}[r] &{M} \ar@{=}[r] &{M}
}
\end{equation*}
is clearly a regular epimorphism.

Thus, an $\mathcal{R}$-full Schreier internal category $\mathcal{E}$ is also aspherical, if we transplant in the above sense the notion of \cite{aspherical} to $cc\text{-}SExt_M,$ but the converse is false, because asking for $\la d_X,c_X\ra\colon X_1\longrightarrow R_{{\mathbb{E}}_0}$ to be a Schreier extension is strictly stronger than asking for it to be a regular epimorphism. Of course, the difference disappears when $M=G$ is a group and we look at $AbExt_G\subseteq cc\text{-}SExt_G.$

Finally, as we have seen, we do need the stronger notion of $\mathcal{R}$-fullness to describe crossed Schreier extensions in terms of internal structures in $cc\text{-}SExt_M$ (Theorem \ref{thm:equivalence_RSCat-XSExt}), which motivates the introduction of Definition \ref{def:R-full}.

Coming back to the main goal of this section, the idea is that the morphism \eqref{eqn:aspherical_in_cc-SExt_M} plays the role which in the general definition of \eqref{eqn:1-direction-bourn}, for a naturally Mal'tsev category $\mathcal{C},$ is played by the regular epimorphism $\xymatrixcolsep{1.5pc}\xymatrix{{\la d_Y,c_Y\ra\colon Y_1}  \ar@{->>}[r] &{Y_0\times Y_0}}$ associated with an aspherical (abelian) groupoid $\xymatrix{{\underline{Y}: Y_1} \ar@<.9ex>[r]^-{d_Y} \ar@<-.9ex>[r]_-{c_Y} &{Y_0.} \ar[l]|-{e_Y}}$

Accordingly, we want to compute the kernel in $\mathbf{CMon}(cc\text{-}SExt_M)\cong cc\text{-}SPt_M\cong can\text{-}\mathcal{S}\mathrm{mod}_M$ of the morphism resulting by the application of the direction functor \eqref{eqn:direction_functor} to \eqref{eqn:aspherical_in_cc-SExt_M}. This can be achieved directly, using the properties of $d.$ Indeed, the desired kernel $\mathbb{P}$ is given by a pullback
\begin{equation}
\label{eqn:ker(d_X,c_X)}
\begin{aligned}
\xymatrix{
{\mathbb{P}} \pullback\ar[r] \ar[d] &{d({\mathbb{E}}_1)} \ar[d]^-{d(\la d_X,c_X \ra)} \\
{{\mathds{1}}_M} \ar@{>->}[r]_{\widetilde{s_0}} &{d(\widehat{R}_{{\mathbb{E}}_0})}
}
\end{aligned}
\end{equation}
in $\mathbf{CMon}(cc\text{-}SExt_M),$ where $\xymatrix{{{\mathds{1}}_M:0} \ar[r] &M \ar@{=}[r] &{M}}$ and $\widetilde{s_0}$ is the unit of the internal monoid $d(\widehat{R}_{{\mathbb{E}}_0})\in\mathbf{CMon}(cc\text{-}SExt_M)$ (cf.~\cite[Theorem 6.5]{AMR}). In particular, $\mathbb{P}$ is then a subobject of $d({\mathbb{E}}_1)=(K_1,\eta_1)\in can\text{-}\mathcal{S}\mathrm{mod}_M,$ meaning that it is determined by the restriction of the action $\eta_1\colon M\rightarrow\mathrm{End}(K_1)$ induced by $\mathbb{E}_1$ to a submonoid $A$ of $K_1$ which is closed under $\eta_1.$ According to \eqref{eqn:ker(d_X,c_X)}, $A$ is given by the pullback
\begin{equation*}
\xymatrix{
A \pullback\ar@{>->}[r] \ar[d] &{K_1} \ar[d]^-{\la d^\star_X,h\ra} \\
{0} \ar[r] &{K_0\times K_0}
}
\end{equation*}
in $\mathbf{Mon},$ and is thus a kernel of $\la d^\star_X,h\ra$: but this is the same as a kernel of $\la d_X,c_X\ra\colon X_1\longrightarrow R_{{\mathbb{E}}_0}$ (because in diagram \eqref{eqn:aspherical_in_cc-SExt_M} the left hand square is a pullback), which is in turn the same as a kernel of the crossed semimodule $\delta=c_Xk_d\colon K[d_X]\longrightarrow X_0$ corresponding by Theorem \ref{thm:equivalence_crossed_semimodules_schreier_categories} to the Schreier internal category $\xymatrix{{X_1} \ar@<.9ex>[r]^-{d_X} \ar@<-.9ex>[r]_-{c_X} &{X_0.} \ar[l]|-{e_X} }$ (Here, as always, $(K[d_X],k_d)$ denotes a kernel of $d_X.$)

Our next step is to show that the above $M$-semimodule structure on $A,$ resulting as the kernel of $d(\la d_X,c_X \ra)\colon d({\mathbb{E}}_1)\longrightarrow d(\widehat{R}_{{\mathbb{E}}_0})$ in $\mathbf{CMon}(cc\text{-}SExt_M),$ coincides with the one induced by the crossed Schreier extension
\begin{equation}
\label{eqn:crossed_ext_2}
\begin{aligned}
\xymatrixrowsep{1.5pc}
\xymatrix{
{A} \ar@{>->}[r]^-{k} &{K[d_X]} \ar[rr]^-{\delta=c_Xk_d} \ar@{->>}[rd]_-g &\ &{X_0} \ar@{->>}[r]^-{f_0} &{M} \\
&\ &{K_0} \ar@{>->}[ru]_-{k_0} &\ &\
}
\end{aligned}
\end{equation}
corresponding to $\mathcal{E}$ in the equivalence of Theorem \ref{thm:equivalence_RSCat-XSExt}.
\begin{proposition}
\label{prop:same_action}
The $M$-action on $A$ induced by \eqref{eqn:crossed_ext_2} is the restriction to $A$ of the $M$-action on $K_1$ induced by the Schreier extension ${\mathbb{E}}_1.$
\end{proposition}
\begin{proof}
There is no loss in generality in realising $A$ and $K[d_X]$ as submonoids of $K_1,$ so that $A\subseteq K[d_X]$ and we can write $k_1(a)$ for elements $a\in K[d_X].$

The action $\eta_1\colon M\times K_1\dashrightarrow K_1,$ $(m,a_1)\mapsto\eta_1(m)(a_1),$ induced by ${\mathbb{E}}_1$ is defined by the condition $u_m+k_1(a_1)=k_1\big(\eta_1(m)(a_1)\big)+u_m,$ where $u_m$ is any representative of $m$ in ${\mathbb{E}}_1.$

At the same time, the $M$-action on $A$ induced by the crossed Schreier extension \eqref{eqn:crossed_ext_2} is obtained as follows (cf. Section 2):
\begin{itemize}
\item let $\alpha$ denote the given action of $X_0$ on $K[d_X]$ in the crossed semimodule $\delta;$ then, since ${\mathbb{E}}_0$ is a $cc$-Schreier extension, $\alpha$ restricts to an action (which we shall still call $\alpha$) of $X_0$ to $A;$
\item next, being $\xymatrixcolsep{1.5pc}\xymatrix{A \ar@{>->}[r]^-k &{K[d_X]} \ar@{->>}[r]^-g &{K_0}}$ also a $cc$-Schreier extension, there results an action $\beta\colon M\times A\dashrightarrow A,$ $\big(m=f_0(x_0),a\big)\mapsto \beta(m)(a)=\alpha(x_0)(a).$
\end{itemize}
Then, since we know from the proof of Theorem \ref{thm:equivalence_crossed_semimodules_schreier_categories} that the $X_0$-action on $K[d_X]$ is the one corresponding to the Schreier point $(d_X,e_X),$ we see that the element $\beta(m)(a)=\alpha(x_0)(a)$ is defined by the property
\begin{equation*}
e_X(x_0)+k_1(a)=k_1\big(\beta(m)(a)\big)+e_X(x_0).
\end{equation*}
Now, consider the $cc$-Schreier extension ${\mathbb{E}}_0,$ and write $x_0=k_0(a_0)+v_m,$ for some representative $v_m$ of $m=f_0(x_0)=f_1e_X(x_0)$: then the above equation can be written as
\begin{equation*}
e_Xk_0(a_0)+e_X(v_m)+k_1(a)=k_1\big(\beta(m)(a)\big)+e_Xk_0(a_0)+e_X(v_m),
\end{equation*}
i.e.
\begin{equation}
\label{eqn:act_2}
k_1e^\star_X(a_0)+e_X(v_m)+k_1(a)=k_1\big(\beta(m)(a)+e^\star_X(a_0)\big)+e_X(v_m).
\end{equation}
Since $e_X\colon{\mathbb{E}}_0\longrightarrow{\mathbb{E}}_1$ is a morphism of Schreier extensions, $e_X(v_m)$ is a representative of $m$ in ${\mathbb{E}}_1,$ so that
\begin{equation*}
	e_X(v_m)+k_1(a)=k_1\big(\eta_1(m)(a)\big)+e_X(v_m).
\end{equation*}
Thus, \eqref{eqn:act_2} gives
\begin{equation*}
k_1\big(e^\star_X(a_0)+\eta_1(m)(a)\big)+e_X(v_m)=k_1\big(\beta(m)(a)+e^\star_X(a_0)\big)+e_X(v_m),
\end{equation*}
and since $e_X(v_m)$ is a representative and $K_1$ is commutative and cancellative, we conclude that $\beta(m)(a)=\eta_1(m)(a).$
\end{proof}
Eventually, using a terminology akin to that of \cite{aspherical, Bourn-Rodelo}, we can define:
\begin{definition}
\label{def:1-direction}
The cancellative $M$-semimodule $(A,\eta_1)$ is the \emph{direction} of the $\mathcal{R}$-full Schreier internal category $\mathcal{E}.$ It will be denoted by $(A,\eta_1)=d_1(\mathcal{E}).$
\end{definition}
By the above, if $\xymatrix{{\mathcal{E}:{{\mathbb{E}}_1}} \ar@<.9ex>[r]^-{d_X} \ar@<-.9ex>[r]_-{c_X} &{{\mathbb{E}}_0} \ar[l]|-{e_X}
}$ is an $\mathcal{R}$-full Schreier internal category in $cc\text{-}SExt_M,$ with object of arrows $\xymatrix{{\mathbb{E}_1:K_1} \ar@{>->}[r]^-{k_1} &{X_1} \ar@{->>}[r]^-{f_1} &{M}}$ and object of objects $\xymatrix{{\mathbb{E}_0:K_0} \ar@{>->}[r]^-{k_0} &{X_0} \ar@{->>}[r]^-{f_0} &{M,}}$ its direction $d_1(\mathcal{E})$ can be described equivalently as:
\begin{itemize}
\item the monoid $A\cong K[d_X]\cap K[c_X]\subseteq K_1$ resulting from the intersection of the kernels of the domain and codomain morphisms of $\xymatrix{{\underline{X}: X_1} \ar@<.9ex>[r]^-{d_X} \ar@<-.9ex>[r]_-{c_X} &{X_0,} \ar[l]|-{e_X}}$ together with the (restriction of the) $M$-action $\eta_1$ on $K_1$ induced by ${\mathbb{E}}_1;$
\item the internal commutative monoid $\xymatrixcolsep{2.5pc}\xymatrix{{A\rtimes_{\eta_1} M} \ar@<.5ex>[r]^-{p_M} &{M} \ar@<.5ex>[l]^-{\la 0,1_M\ra}}$ in $cc\text{-}SExt_M,$ where $A$ is a kernel of $\la d_X,c_X\ra\colon X_1\longrightarrow R_{{\mathbb{E}}_0}$ and $\eta_1$ is the (restriction of the) action induced by ${\mathbb{E}}_1;$
\item the kernel of the crossed Schreier extension \eqref{eqn:crossed_ext_2}, with the action described in Proposition~\ref{prop:same_action}, corresponding to $\mathcal{E}$ by Theorem \ref{thm:equivalence_RSCat-XSExt}.
\end{itemize}
\begin{remarks}
\label{rmk:d_1-is-d_1}
\begin{enumerate}
\item Let $\xymatrix{
{\mathcal{E}:{{\mathbb{E}}_1}} \ar@<.9ex>[r]^-{d_X} \ar@<-.9ex>[r]_-{c_X} &{{\mathbb{E}}_0} \ar[l]|-{e_X}
}$ be as above. By \eqref{eqn:aspherical_in_cc-SExt_M} and Remark \ref{rmks:R-hat}$(3),$ we have for the induced morphism $\la d_X,c_X\ra$ a commutative triangle
\begin{equation}
\label{eqn:fact_<d_X,c_X>}
\begin{aligned}
\xymatrixcolsep{1.3pc}
\xymatrixrowsep{1.5pc}
\xymatrix{
{{\mathbb{E}}_1} \ar[rr]^-{\la d_X,c_X\ra} \ar[rd]_-{\la d_X,c_X\ra} &\ &{{\mathbb{E}}_0\times{\mathbb{E}}_0} \\
&{\widehat{R}_{{\mathbb{E}}_0}.} \ar@{>->}[ru]_-{j_0} &\
}
\end{aligned}
\end{equation}
Since $d$ preserves finite products and monomorphisms, by applying $d$ to \eqref{eqn:fact_<d_X,c_X>} we get a factorisation
\begin{equation*}
\xymatrixcolsep{1.3pc}
\xymatrix{
{d({\mathbb{E}}_1)} \ar[rr]^-{\la d(d_X),d(c_X)\ra} \ar[rd]_-{d(\la d_X,c_X)\ra} &\ &{d({\mathbb{E}}_0)\times d({\mathbb{E}}_0)} \\
&{d(\widehat{R}_{{\mathbb{E}}_0})} \ar@{>->}[ru]_-{d(j_0)} &\
}
\end{equation*}
which shows that when $j_0$ is an isomorphism (i.e., $K_0$ is a group) and $\mathcal{E}$ is $\mathcal{R}$-full, then $d(\la d_X,c_X\ra)$ is precisely the regular epimorphism
\begin{equation*}
\xymatrixcolsep{1pc}\xymatrix{{\la d(d_X),d(c_X)\ra\colon d({\mathbb{E}}_1)} \ar@{->>}[r] &{d({\mathbb{E}}_0)\times d({\mathbb{E}}_0).}}
\end{equation*}
This fact, together with Remark~\ref{rmk:aspherical_groupoids}, shows that, when $M=G$ is a group, the direction $d_1(\mathcal{E})$ of an $\mathcal{R}$-full Schreier category in $AbExt_G\subseteq cc\text{-}SExt_G$ is the same as the direction of an aspherical groupoid in $\mathcal{C}=\mathrm{Mal}(\mathbf{Gp}/G),$ in the sense of \cite{aspherical, Bourn-Rodelo}.
\item We have seen that the actual computation of $d(\la d_X,c_X\ra)$ is not needed to construct $d_1(\mathcal{E}).$ Nonetheless, it is not difficult to prove that for every $\mathbb{E}\in cc\text{-}SExt_M$ one has $d\Big(\widehat{R}_{\mathbb{E}}\Big)\cong\widehat{R}_{d(\mathbb{E})}$ as Schreier points on $M,$ because both points correspond to the same $M$-action on $K\times K$ (which is component-wise the action induced by $\mathbb{E},$ cf. Remark~\ref{rmks:R-hat}$(3)$), so that one can compute $d(\la d_X,c_X\ra)\colon d({\mathbb{E}}_1)\longrightarrow d(\widehat{R}_{{\mathbb{E}}_0})$ explicitly as
\begin{equation*}
d(\la d_X,c_X\ra)=\la d(d_X),d(c_X) \ra\colon d({\mathbb{E}}_1)\longrightarrow \widehat{R}_{d({\mathbb{E}}_0)}.
\end{equation*}
\end{enumerate}
\end{remarks}

We now address the question of the functoriality of
\begin{equation}
\resizebox{0.9\textwidth}{!}
{\label{eqn:new_functor_d_1}
$d_1\colon\mathcal{R}\mathrm{SCat}(cc\text{-}SExt_M)\longrightarrow \mathbf{CMon}(cc\text{-}SExt_M)\cong cc\text{-}SPt_M\cong can\text{-}\mathcal{S}\mathrm{mod}_M,$}
\end{equation}
defined on the objects as in Definition \ref{def:1-direction}: to do that, we shall first look at the wider construction \eqref{eqn:action_monoids}.
\begin{proposition}
\label{prop:functor_D_1}
The process of associating to a crossed Schreier extension (in the general sense of Definition \ref{def:crossed_Schreier_extension})
\begin{equation*}
\xymatrix{
A \ar@{>->}[r]^-k &B \ar[r]^-{\delta} &C \ar@{->>}[r]^-{\pi} &M
}
\end{equation*}
its induced $M$-action on $A$ is functorial, yielding a functor
\begin{equation}
D_1\colon \mathrm{XSExt}_M\longrightarrow can\text{-}\mathcal{S}\mathrm{mod}_M.
\end{equation}
\end{proposition}
(Observe that the difference with $com\text{-}\mathrm{XSExt}_M$ is that in this case we are not assuming $B$ to be commutative, though the commutativity of the kernel $A$ is always guaranteed by the axioms of the crossed semimodule $\delta.$)
\begin{proof}
Recall that a morphism in $\mathrm{XSExt}_M$ is given by a commutative diagram
\begin{equation*}
\xymatrix{
A \ar@{>->}[r]^-k \ar[d]_-{\varphi^\star} &B \ar[r]^-{\delta} \ar[d]_-{\varphi} &C \ar@{->>}[r]^-{\pi} \ar[d]^-{\psi} &M\ar@{=}[d]  \\
{A'} \ar@{>->}[r]_-{k'} &{B'} \ar[r]_-{\delta'} &{C'} \ar@{->>}[r]_-{\pi'} &M
}
\end{equation*}
in which $(\varphi,\psi)$ is a morphism of crossed semimodules, and such that in the factorisation
\begin{equation*}
\xymatrixrowsep{1.3pc}
\xymatrix{
A \ar[dd]_-{\varphi^\star} \ar@{>->}[r]^-k &B \ar[dd]_-{\varphi} \ar@{->>}[rd]_-{\rho} \ar[rr]^-{\delta} & \ &C  \ar[dd]_-{\psi} \ar@{->>}[r]^-{\pi} &{M}  \ar@{=}[dd] \\
 &\ &N \ar[dd]_(.20){\psi^\star} \ar@{>->}[ru]_-{\nu}&\ &\ &\ \\
{A'} \ar@{>->}[r]^-{k'} &{B'} \ar@{->>}[rd]_-{\rho'} \ar[rr]^(.70){\delta'} & \ &{C'} \ar@{->>}[r]^-{\pi'} &{M} \\
 &\ &{N'} \ar@{>->}[ru]_-{\nu'}&\ &\ &\
}
\end{equation*}
both $(\varphi^\star,\varphi,\psi^\star)$ and $(\psi^\star,\psi,1_M)$ are morphisms of Schreier extensions.

Then, denoting by $\beta\colon M\longrightarrow\mathrm{End}(A)$ and $\beta'\colon M\longrightarrow\mathrm{End}(A')$ the induced actions, the point is to show that $\varphi^\star\colon (A,\beta)\longrightarrow(A',\beta')$ is a morphism in $\mathcal{S}\mathrm{mod}_M,$ i.e. that $\varphi^\star$ is equivariant under the actions: the functoriality of $D_1$ will follow immediately.

We know that for every $m\in M$ and $a\in A$ we can write $\beta(m)(a)=\alpha(c)(a),$ where $\alpha\colon C\longrightarrow\mathrm{End}(A)$ is the action given by the crossed semimodule $\delta$ (using the fact that $\pi$ is a Schreier extension and $N$ is cancellative) and $c\in C$ is any element such that $\pi(c)=m.$

But then, if $\pi(c)=m,$ we also have $m=\pi'\big(\psi(c)\big),$ so that (realising $k$ and $k'$ as inclusions, for simplicity)
\begin{equation*}
\begin{split}
\beta'(m)\big(\varphi^\star(a)\big)&=\alpha'\big(\psi(c)\big)\big(\varphi^\star(a)\big)\\
&=\varphi^\star\big(\alpha(c)(a)\big)\\
&=\varphi^\star\big(\beta(m)(a)\big),
\end{split}
\end{equation*}
and we are done.

Observe that we have used the fact that since the right-hand side diagram in \eqref{eqn:mor_crossed_semimod} commutes by assumption, and $\alpha,$ $\alpha'$ restrict to actions on $A$ and $A',$ the diagram
\begin{equation*}
\xymatrix{
{C\times A} \ar[d]_-{\psi\times\varphi^\star} \ar@{-->}[r]^-{\alpha} &A \ar[d]^-{\varphi^\star} \\
{C'\times A'} \ar@{-->}[r]_-{\alpha'} &{A'}
}
\end{equation*}
is also commutative.
\end{proof}
Of course, then, the functor $D_1$ restricts to a functor ${D_1}_{\vert}$ on the full subcategory $com\text{-}\mathrm{XSExt}_M\subseteq \mathrm{XSExt}_M.$

We can thus define $d_1$ in \eqref{eqn:new_functor_d_1} as the only functor making the following commute:
\begin{equation}
\label{def:d_1}
\begin{aligned}
\xymatrixcolsep{2.5pc}
\xymatrixrowsep{3pc}
\xymatrix{
{\mathcal{R}\mathrm{SCat}(cc\text{-}SExt_M)} \ar[d]^-{\wr}_-{F} \ar[r]^-{d_1} &{\mathbf{CMon}(cc\text{-}SExt_M)} \\
{com\text{-}\mathrm{XSExt}_M} \ar[r]_-{{D_1}_{\vert}} &{can\text{-}\mathcal{S}\mathrm{mod}_M,} \ar[u]_-{\wr}
}
\end{aligned}
\end{equation}
where $F$ (as in \eqref{eqn:F}) is the equivalence of Theorem \ref{thm:equivalence_RSCat-XSExt} and the equivalence $can\text{-}\mathcal{S}\mathrm{mod}_M\cong\mathbf{CMon}(cc\text{-}SExt_M)$ $(~\cong cc\text{-}SPt_M)$ is the content of \cite[Theorem 6.5]{AMR} (see also the discussion preceding Proposition 6.10 in the same \cite{AMR}).

Observe that it follows by the equivalent descriptions for the direction of an $\mathcal{R}$-full Schreier internal category $\mathcal{E}$ in $cc\text{-}SExt_M,$ that, at the level of objects, the functor $d_1$ defined by means of the commutative diagram \eqref{def:d_1} provides exactly the direction $d_1(\mathcal{E})$ in the sense of Definition \ref{def:1-direction}; and indeed, since $d_1$ and the restricted functor ${D_1}_{\vert}$ are equivalent by construction, from now on we shall denote also the latter as $d_1.$
\begin{definition}
The functor $d_1$ in \eqref{eqn:new_functor_d_1} is called the $1$\emph{-dimensional direction functor} on $cc\text{-}SExt_M.$
\end{definition}
By Remark \ref{rmk:d_1-is-d_1}$(1),$ as was our goal, this functor extends to $cc$-Schreier extensions of monoids the classical $1$-dimensional direction functor
\begin{equation*}
d_1\colon\mathrm{AAGpd}\big(\mathrm{Mal}(\mathcal{C})\big)\longrightarrow\mathbf{Ab}\big(\mathrm{Mal}(\mathcal{C})\big)
\end{equation*}
of \cite{aspherical, Bourn-Rodelo}, for the slice categories $\mathcal{C}=\mathbf{Gp}/G$ of the category of groups.
\section{Properties of the functors $D_1$ and $d_1$}
We shall now study the main properties of the functor
\begin{equation*}
D_1\colon \mathrm{XSExt}_M\longrightarrow can\text{-}\mathcal{S}\mathrm{mod}_M,
\end{equation*}
as defined in Proposition \ref{prop:functor_D_1}, and of its restriction
\begin{equation*}
d_1\colon com\text{-}\mathrm{XSExt}_M\cong \mathcal{R}\mathrm{SCat}(cc\text{-}SExt_M\longrightarrow \mathbf{CMon}(cc\text{-}SExt_M)\cong can\text{-}\mathcal{S}\mathrm{mod}_M.
\end{equation*}

Our goal is to show that these functors satisfy the conditions of the following well known general theorem:
\begin{theorem}
\label{thm:cofibration_monoidal}
Let $\mathcal{U}$ be a category with finite products and ${F}\colon\mathcal{U}\longrightarrow\mathcal{V}$ a cofibration which preserves finite products. Suppose also that the cocartesian morphisms in $\mathcal{U}$ are stable under finite products. Then, if $(M,\omega\colon M\times M\longrightarrow M,\varepsilon\colon1\longrightarrow M)$ is an internal monoid in $\mathcal{V},$ there results a monoidal structure on the fibre ${F}^{-1}(M),$ which is symmetric as soon as the internal monoid $M$ is commutative.
\end{theorem}
(See for example \cite{bourn-direction} for a proof.)

Accordingly, we shall deduce that the fibres of $D_1$ and $d_1$ are naturally symmetric monoidal categories, and we shall investigate the commutative monoid structure resulting on the sets of their connected components. This will allow us to take a step forward in the construction of the cohomology theory for monoids, with coefficients in semimodules, which began in \cite{AMR}.

First, we need to understand how finite products are defined in $\mathrm{XSExt}_M.$

As for the empty product, it is immediate that for any monoid $M$ the trivial sequence
\begin{equation*}
\xymatrix{
0 \ar@{=}[r] &0 \ar[r] &M \ar@{=}[r]  &M
}
\end{equation*}
is a terminal object in both $\mathrm{XSExt}_M$ and $com\text{-}\mathrm{XSExt}_M,$ which is then mapped by $D_1$ and $d_1$ to the trivial $M$-semimodule $0$ (i.e., to a terminal object in $ can\text{-}\mathcal{S}\mathrm{mod}_M$).

Next, suppose that we are given two crossed Schreier extensions
\begin{equation*}
\xymatrixrowsep{1.5pc}
\xymatrix{
{\mathcal{E}:A} \ar@{>->}[r]^-k &B \ar@{->>}[rd]_-{\rho} \ar[rr]^-{\delta} & \ &C \ar@{->>}[r]^-{\pi} &{M} \\
 &\ &N \ar@{>->}[ru]_-{\nu}&\ &\ &\
}
\end{equation*}
and
\begin{equation*}
\xymatrixrowsep{1.5pc}
\xymatrix{
{\mathcal{E}':A'} \ar@{>->}[r]^-{k'} &{B'} \ar@{->>}[rd]_-{\rho'} \ar[rr]^-{\delta'} & \ &{C'} \ar@{->>}[r]^-{\pi'} &{M} \\
 &\ &{N'} \ar@{>->}[ru]_-{\nu'}&\ &\ &\
}
\end{equation*}
and denote by $\alpha\colon C\longrightarrow \mathrm{End}(B)$ and  $\alpha'\colon C'\longrightarrow \mathrm{End}(B')$ the actions in the crossed semimodules $\delta$ and $\delta',$ respectively.

Then $\delta\times\delta'\colon C\times C'\longrightarrow B\times B'$ is again a crossed semimodule, under the product action
\begin{equation}
\label{eqn:prodact1}
\xymatrix{ {\alpha\times\alpha'\colon C\times C'} \ar[r] &{\mathrm{End}(B)\times \mathrm{End}(B')} \ar@{>->}[r] &{\mathrm{End}(B\times B'),}}
\end{equation}
and
\begin{equation*}
\xymatrixrowsep{1.5pc}
\xymatrix{
{A\times A'} \ar@{>->}[r]^-{k\times k'} &{B\times B'} \ar@{->>}[rd]_-{\rho\times\rho'} \ar[rr]^-{\delta\times\delta'} & \ &{C\times C'} \ar@{->>}[r]^-{\pi\times \pi'} &{M\times M} \\
 &\ &{N\times N'} \ar@{>->}[ru]_-{\nu\times \nu'}&\ &\ &\
}
\end{equation*}
is a crossed Schreier extension.

By \cite{AMMPS}, pulling back $\pi\times \pi'$ along the diagonal $\Delta_M=\la 1_M,1_M\ra$ yields a commutative diagram
\begin{equation}
\label{eqn:prod_x}
\begin{aligned}
\xymatrixrowsep{1.3pc}
\xymatrix{
{A\times A'} \ar@{=}[dd] \ar@{>->}[r]^-{k\times k'} &{B\times B'} \ar@{=}[dd] \ar@{->>}[rd]_-{\rho\times\rho'} \ar[rr]^-{\delta''} & \ &{C''} \ar@{->>}[r]^-{\pi''} \pullback \ar@{>->}[dd]_-{\iota''} &{M} \ar@{>->}[dd]^-{\Delta_M} \\
 &\ &{N\times N'} \ar@{=}[dd]  \ar@{>->}[ru]_-{\nu''}&\ &\ &\ \\
 {A\times A'} \ar@{>->}[r]^-{k\times k'} &{B\times B'} \ar@{->>}[rd]_-{\rho\times\rho'} \ar[rr]^(.70){\delta\times\delta'} & \ &{C\times C'} \ar@{->>}[r]^-{\pi\times \pi'} &{M\times M} \\
 &\ &{N\times N'} \ar@{>->}[ru]_-{\nu\times\nu'}&\ &\ &\
}
\end{aligned}
\end{equation}
in which the upper row, as well as the lower one, is a crossed Schreier extension. The action of $C''$ on $B\times B'$ in the crossed semimodule $\delta''$ is given by the composition
\begin{equation}
\label{eqn:prodact2}
(\alpha\times \alpha')\iota''\colon C''\longrightarrow\mathrm{End}(B\times B),
\end{equation}
and the whole diagram \eqref{eqn:prod_x} is actually a morphism of crossed Schreier extensions, in the sense of \cite{AMMPS}.

Observe that $C''$ can be realised as the pullback
\begin{equation*}
\xymatrix{
{C''} \ar[r]^-{\pi_2} \ar[d]_-{\pi_1} \pullback &{C'} \ar[d]^-{\pi'} \\
{C} \ar[r]_-{\pi} &{M,}
}
\end{equation*}
by which description one has $\iota''=\la\pi_1,\pi_2 \ra,$ $\pi''=\pi'\pi_2=\pi\pi_1,$ $\nu''=\nu\times_M\nu'$ and $\delta''=\delta\times_M\delta'.$

It is straightforward, then, that
\begin{equation}
\label{eqn:product_XES}
\xymatrix{
{A\times A'} \ar@{>->}[r]^-{k\times k'} &{B\times B'} \ar[r]^-{\delta''} &{C''} \ar@{->>}[r]^-{\pi''} &{M}
}
\end{equation}
is a product of $\mathcal{E}$ and $\mathcal{E}'$ in $\mathrm{XSExt}_M,$ with projections
\begin{equation*}
\begin{aligned}
{\xymatrixcolsep{0.1pc}
\xymatrixrowsep{2.35pc}
\xymatrix{
{\mathcal{E}\times\mathcal{E}'} \ar[dd]_-{p_1} &{:}\\
{} \\
{\mathcal{E}} &{:}
}}
{\xymatrix{
{A\times A'} \ar@{=}[d] \ar@{>->}[r]^-{k\times k'} &{B\times B'} \ar[r]^-{\delta''} \ar@{=}[d] &{C''} \ar@{->>}[r]^-{\pi''} \ar[d]\pullback \ar@{>->}[d]_-{\iota''} &{M} \ar@{>->}[d]^-{\Delta_M} \\
 {A\times A'} \ar[d]_-{p_1} \ar@{>->}[r]^-{k\times k'} &{B\times B'} \ar[d]^-{p_1} \ar[r]^-{\delta\times\delta'} &{C\times C'} \ar[d]^-{p_1} \ar@{->>}[r]^-{\pi\times \pi'} &{M\times M} \ar[d]^-{p_1} \\
 {A} \ar@{>->}[r]_-{k} &{B} \ar[r]_-{\delta} &{C} \ar@{->>}[r]_-{\pi} &{M}
}}
\end{aligned}
\end{equation*}
and
\begin{equation*}
\begin{aligned}
{\xymatrixcolsep{0.1pc}
\xymatrixrowsep{2.35pc}
\xymatrix{
{\mathcal{E}\times\mathcal{E}'} \ar[dd]_-{p_2} &{:}\\
{} \\
{\mathcal{E}'} &{:}
}}
{\xymatrix{
{A\times A'} \ar@{=}[d] \ar@{>->}[r]^-{k\times k'} &{B\times B'} \ar[r]^-{\delta''} \ar@{=}[d] &{C''} \ar@{->>}[r]^-{\pi''} \ar[d]\pullback \ar@{>->}[d]_-{\iota''} &{M} \ar@{>->}[d]^-{\Delta_M} \\
 {A\times A'} \ar[d]_-{p_2} \ar@{>->}[r]^-{k\times k'} &{B\times B'} \ar[d]^-{p_2} \ar[r]^-{\delta\times\delta'} &{C\times C'} \ar[d]^-{p_2} \ar@{->>}[r]^-{\pi\times \pi'} &{M\times M} \ar[d]^-{p_2} \\
 {A'} \ar@{>->}[r]_-{k'} &{B'} \ar[r]_-{\delta'} &{C'} \ar@{->>}[r]_-{\pi'} &{M.}
}}
\end{aligned}
\end{equation*}
Since $B\times B'$ is commutative if so are $B$ and $B',$ the product \eqref{eqn:product_XES} belongs to $com\text{-}\mathrm{XSExt}_M$ as soon as both $\mathcal{E}$ and $\mathcal{E}'$ are in $com\text{-}\mathrm{XSExt}_M$: by this, and what was said above about the terminal object, we conclude that $com\text{-}\mathrm{XSExt}_M$ is closed under finite products in $\mathrm{XSExt}_M.$

Finally, it follows by \eqref{eqn:prodact2} and \eqref{eqn:prodact1} that the $M$-action on $A\times A'$ induced by \eqref{eqn:product_XES} is given component-wise by the actions induced by $\mathcal{E}$ and $\mathcal{E}',$ and we have:
\begin{proposition}
The functors $D_1$ and $d_1$ preserve finite products.
\end{proposition}
We come to the question of determining the cocartesian morphisms for $D_1$ and $d_1.$

Recall that an arrow $f\colon Y\longrightarrow X$ in a category $\mathcal{U}$ is \emph{cocartesian} for a functor $F\colon\mathcal{U}\longrightarrow\mathcal{V}$ (briefly, $f$ is $F$\emph{-cocartesian}) if for every $g\colon Y\longrightarrow Z$ in $\mathcal{U}$ and for every $v\colon F(X)\longrightarrow F(Z)$ in $\mathcal{V}$ such that $F(g)=v F(f),$ there exists a unique $h_{v}\colon X\longrightarrow Z$ satisfying $F(h_{v})=v$ and $h_{v}f=g.$ As an immediate consequence, if $f\colon Y\longrightarrow X$ and $f'\colon Y\longrightarrow X'$ are both $F$-cocartesian and share the same domain $Y,$ one must have $f'=gf$ for a unique isomorphism $g\colon X\xlongrightarrow{\sim}X'$ in $\mathcal{U}.$

Now, it is proven in \cite{AMMPS} that for every crossed Schreier extension
\begin{equation*}
\xymatrix{
{\mathcal{E}:A} \ar@{>->}[r]^-k &B \ar[r]^-{\delta} &C \ar@{->>}[r]^-{\pi} &{M} \\
}
\end{equation*}
and for every morphism $f\colon(A,\eta)\longrightarrow(A',\eta')$ in $can\text{-}\mathcal{S}\mathrm{mod}_M,$ where $\eta$ is the $M$-action induced on $A$ by $\mathcal{E}$ (so that $(A,\eta)=D_1(\mathcal{E})$), there is a universal construction, called the \emph{push forward} of $\mathcal{E}$ along $f,$ which allows to expand $f$ into a morphism
\begin{equation}
\label{eqn:pushforward}
\xymatrixrowsep{1.3pc}
\begin{aligned}
\xymatrix{
A \ar[dd]_-{f} \ar@{>->}[r]^-k &B \ar[dd]_-{\vartheta_f} \ar@{->>}[rd]_-{\rho} \ar[rr]^-{\delta} & \ &C  \ar@{=}[dd] \ar@{->>}[r]^-{\pi} &{M}  \ar@{=}[dd] \\
 &\ &N \ar@{=}[dd] \ar@{>->}[ru]_-{\nu}&\ &\ &\ \\
{A'} \ar@{>->}[r]^-{k_f} &{Q_f} \ar@{->>}[rd]_-{\rho_f} \ar[rr]^(.70){\delta_f} & \ &{C} \ar@{->>}[r]^-{\pi} &{M} \\
 &\ &{N} \ar@{>->}[ru]_-{\nu}&\ &\ &\
}
\end{aligned}
\end{equation}
of crossed Schreier extensions. The codomain
\begin{equation*}
\xymatrix{
{\mathcal{E}_f:A'} \ar@{>->}[r]^-{k_f} &{Q_f} \ar[r]^-{\delta_f} &C \ar@{->>}[r]^-{\pi} &{M}
}
\end{equation*}
of this morphism is obtained by considering a coequaliser $(Q_f,q)$ of $\la0,1_B\ra k=\la 0,k\ra$ and $\la 1_{A'},0\ra f=\la f,0\ra$
\begin{equation}
\label{eqn:costr_Q}
\begin{aligned}
\xymatrix{
A \ar[d]_-{f} \ar[r]^-k &{B} \ar[d]^-{\la0,1_B\ra} \ar@/^1.5pc/@{.>}[rdd]^-{\vartheta_f} &\ \\
{A'} \ar@/_1.5pc/@{.>}[rrd]_-{k_f} \ar[r]_-{\la1_{A'},0\ra} &{A'\times B} \ar@{->>}[rd]^-{q} &\ \\
&\ &{Q_f}
}
\end{aligned}
\end{equation}
and the crossed semimodule $\delta_f\colon Q_f\longrightarrow C$ results by the equalities $\delta p_B\la0,k\ra=\delta k=0=\delta p_B\la f,0\ra$ (where, of course, $p_B$ denotes the projection on $B$) and the universal property of the coequaliser.
\begin{remark}
If in the above construction one has $\mathcal{E}\in com\text{-}\mathrm{XSExt}_M,$ so that the monoid $B$ is both commutative and cancellative (by Lemma \ref{lemma:B_cancellative}$(1)$), the $C$-action $\alpha$ on $B$ in the crossed semimodule $\delta$ yields a well defined action $\omega\colon M\longrightarrow\mathrm{End}(B),$ $\omega(m)(b)=\alpha(c)(b),$ where $c\in C$ is any element such that $\pi(c)=m.$

Then the product $A'\times B$ in \eqref{eqn:costr_Q} is actually a biproduct of $(A',\eta')$ and $(B,\omega)$ in $can\text{-}\mathcal{S}\mathrm{mod}_M,$ and it follows that the push forward $\mathcal{E}_f$ is also in $ com\text{-}\mathrm{XSExt}_M,$ and that in the diagram \eqref{eqn:pushforward} the left hand square is a (pullback and a) pushout.
\end{remark}

The main point, for us, is that the universal property of the push forward, as discussed in \cite{AMMPS}, consists precisely in the fact that the morphism \eqref{eqn:pushforward} is cocartesian for $D_1.$ This gives us the foremost property of $D_1$:
\begin{proposition}
The functor $D_1$ is a cofibration whose cocartesian morphisms are the morphisms
\begin{equation}
\label{eqn:coc_mor}
\begin{aligned}
{\xymatrixcolsep{0.1pc}
\xymatrix{
{\mathcal{E}} \ar[d]_-{(\varphi,\psi)} &{:}\\
{\mathcal{E}'} &{:}
}}
{\xymatrix{
{A} \ar@{>->}[r]^-k \ar[d]_-{\varphi^\star} &B \ar[r]^-{\delta} \ar[d]_-{\varphi} &C \ar@{->>}[r]^-{\pi} \ar[d]^-{\psi} &M\ar@{=}[d]  \\
{A'} \ar@{>->}[r]_-{k'} &{B'} \ar[r]_-{\delta'} &{C'} \ar@{->>}[r]_-{\pi'} &M
}}
\end{aligned}
\end{equation}
in $\mathrm{XSExt}_M$ such that $\psi$ is an isomorphism of monoids.
\end{proposition}
\begin{proof}
The fact that the construction of the cocartesian morphism \eqref{eqn:pushforward} can be carried out for every morphism $f\colon D_1(\mathcal{E})\longrightarrow(A',\eta')$ in the codomain $can\text{-}\mathcal{S}\mathrm{mod}_M$ of $D_1$ guarantees that $D_1$ is a cofibration.

Next, if \eqref{eqn:coc_mor} is cocartesian for $D_1$ and $(\vartheta_{\varphi^\star},1_C)\colon\mathcal{E}\longrightarrow\mathcal{E}_{\varphi^\star}$ is the push forward \eqref{eqn:pushforward} of $\mathcal{E}$ along $\varphi^\star=D_1(\varphi,\psi),$ the fact that $(\vartheta_{\varphi^\star},1_C)$ is also cocartesian on $\varphi^\star$ implies that there is a unique isomorphism $(\varphi',\psi')$
\begin{equation*}
\xymatrixrowsep{2.5pc}
\xymatrixcolsep{2.5pc}
\xymatrix{
{\mathcal{E}} \ar[r]^-{(\varphi,\psi)} \ar[rd]_-{(\vartheta_{\varphi^\star},1_C)} &{\mathcal{E}'} \ar@{.>}[d]^-{\exists!(\varphi',\psi')}_-{\wr} \\
&{\mathcal{E}_{\varphi^\star}}
}
\end{equation*}
in $\mathrm{XSExt}_M$ such that $(\vartheta_{\varphi^\star},1_C)=(\varphi',\psi')\circ(\varphi,\psi)=(\varphi'\varphi,\psi'\psi),$ so that $\psi={\psi'}^{-1}$ is an isomorphism in $\mathbf{Mon}.$

Conversely, suppose that in \eqref{eqn:coc_mor} the monoid homomorphism $\psi$ is an isomorphism, and consider the unique factorisation $(\varphi,\psi)=(\varphi',\psi')\circ(\vartheta_{\varphi^\star},1_C)$
\begin{equation*}
\centering
\xymatrixcolsep{2.5pc}
\xymatrix{
&{\mathcal{E}} \ar[r]^-{(\vartheta_{\varphi^\star},1_C)} \ar[d]_-{(\varphi,\psi)} &{\mathcal{E}_{\varphi^\star}} \ar@{.>}[dl]^{\exists! (\varphi',\psi')}\\
&{\mathcal{E}'}  &\
}
\xymatrix{
&\ \ar@{--}[d] \\
&\
}
\xymatrix{
&{(A,\eta)}  \ar[d]_-{\varphi^\star}  \ar[r]^-{\varphi^\star} &{(A',\eta')}  \ar@{=}[dl] \\
&{(A',\eta')} &\
}
\end{equation*}
through the push forward $\mathcal{E}_{\varphi^\star}.$ Then both $\psi'=\psi$ and the induced morphisms $\psi^\star$ and ${\varphi'}^\star=1_{A'}$
\begin{equation*}
\xymatrixrowsep{1.3pc}
\xymatrix{
{A'} \ar@{=}[dd] \ar@{>->}[r]^-{k_{\varphi^\star}} &{Q_{\varphi^\star}} \ar[dd]_-{\varphi'} \ar@{->>}[rd]_-{\rho_{\varphi^\star}} \ar[rr]^-{\delta_{\varphi^\star}} & \ &C  \ar[dd]_-{\psi}^-{\wr} \ar@{->>}[r]^-{\pi} &{M}  \ar@{=}[dd] \\
 &\ &N \ar[dd]_(.25){\psi^\star}^(.25){\wr} \ar@{>->}[ru]_-{\nu}&\ &\ &\ \\
{A'} \ar@{>->}[r]^-{k'} &{B'} \ar@{->>}[rd]_-{\rho'} \ar[rr]^(.70){\delta'} & \ &{C'} \ar@{->>}[r]^-{\pi'} &{M,} \\
 &\ &{N'} \ar@{>->}[ru]_-{\nu'}&\ &\ &\
}
\end{equation*}
are isomorphisms.

Since by definition of morphism in $\mathrm{XSExt}_M$ the diagram
\begin{equation*}
\xymatrix{
{A'} \ar@{>->}[r]^-{k} \ar@{=}[d] &{Q_{\varphi^\star}} \ar@{->>}[r]^-{\rho_{\varphi^\star}} \ar[d]^-{\varphi'} &{N} \ar[d]_-{\wr}^-{\psi^\star}\\
{A'} \ar@{>->}[r]_-{k'} &{B'} \ar@{->>}[r]_-{\rho'} &{N'}
}
\end{equation*}
is a morphism of Schreier extensions, by the Short Five Lemma for Schreier extensions of monoids (\cite[Proposition 4.5]{P-II}) $\varphi'$ is also an isomorphism.

Accordingly, $(\varphi',\psi')$ is an isomorphism in $\mathrm{XSExt}_M,$ and since $(\vartheta_{\varphi^\star},1_C)$ is cocartesian for $D_1$ and $(\varphi,\psi)=(\varphi',\psi')\circ(\vartheta_{\varphi^\star},1_C),$ we conclude that the same is true for $(\varphi,\psi).$
\end{proof}
By the equivalence $com\text{-}\mathrm{XSExt}_M\cong \mathcal{R}\mathrm{SCat}(cc\text{-}SExt_M)$ of \eqref{eqn:G}, we derive the analogous result for $d_1$:
\begin{corollary}
The direction functor $d_1\colon\mathcal{R}\mathrm{SCat}(cc\text{-}SExt_M)\longrightarrow can\text{-}\mathcal{S}\mathrm{mod}_M$ is a cofibration. A morphism in $\mathcal{R}\mathrm{SCat}(cc\text{-}SExt_M)$ is $d_1$-cocartesian if and only if it is $\mathcal{O}$-invertible, where $\mathcal{O}$ is the ``object of objects'' functor \eqref{eqn:O_R}.
\end{corollary}
Observe that if
\begin{equation*}
\xymatrix{
A \ar@{>->}[r]^-k \ar[d]_-{\varphi^\star} &B \ar[r]^-{\delta} \ar[d]_-{\varphi} &C \ar@{->>}[r]^-{\pi} \ar[d]_-{\wr}^-{\psi} &M\ar@{=}[d]  \\
{A'} \ar@{>->}[r]_-{k'} &{B'} \ar[r]_-{\delta'} &{C'} \ar@{->>}[r]_-{\pi'} &M
}
\end{equation*}
and
\begin{equation*}
\xymatrix{
K \ar@{>->}[r]^-l \ar[d]_-{\vartheta^\star} &U \ar[r]^-{\sigma} \ar[d]_-{\vartheta} &V \ar@{->>}[r]^-{\tau} \ar[d]_-{\wr}^-{\zeta} &M\ar@{=}[d]  \\
{K'} \ar@{>->}[r]_-{l'} &{U'} \ar[r]_-{\sigma'} &{V'} \ar@{->>}[r]_-{\tau'} &M
}
\end{equation*}
are cocartesian morphisms (for $D_1$ or $d_1$), then their product $(\varphi,\psi)\times(\vartheta,\zeta)$ in $(com\text{-})\mathrm{XSExt}_M$ is given by
\begin{equation*}
\xymatrix{
{A\times K} \ar@{>->}[r]^-{k\times l} \ar[d]_-{\varphi^\star\times\vartheta^\star} &{B\times U} \ar[r]^-{\delta''} \ar[d]_-{\varphi\times\vartheta} &P \ar@{->>}[r]^-{p} \ar[d]_-{\wr}^-{\xi} &M\ar@{=}[d]  \\
{A'\times K'} \ar@{>->}[r]_-{k'\times l} &{B'\times U'} \ar[r]_-{\sigma''} &{P'} \ar@{->>}[r]_-{p'} &{M,}
}
\end{equation*}
where the isomorphism $\xi\colon P\longrightarrow P'$ results from the pullbacks
\begin{equation*}
\xymatrix{
P \pullback \ar@{.>}[rd]|{\xi}\ar[r]^-{\sim} \ar[d]_-{\wr} &{\cdot} \ar[r] \ar[d]_-{\wr} \pullback &{V} \ar@/^1.5pc/[dd]^-{\tau} \ar[d]_-{\wr}^-{\zeta} \\
{\cdot} \ar[d] \ar[r]^-{\sim}  \pullback &{P'} \ar[d]  \ar[r] \pullback &{V'} \ar[d]^-{\tau'} \\
C \ar@/_1.5pc/[rr]_-{\pi} \ar[r]^-{\sim}_-{\psi} &{C'} \ar[r]_-{\pi'} &{M.}
}
\end{equation*}
Accordingly, $(\varphi,\psi)\times(\vartheta,\zeta)$ is also cocartesian, and we have:
\begin{proposition}
The cocartesian morphisms for $D_1$ (resp., for $d_1$) are stable under finite products.
\end{proposition}
Thus both functors $D_1,$ $d_1$ fall under the scope of Theorem \ref{thm:cofibration_monoidal}, and since $can\text{-}\mathcal{S}mod_M\cong\mathbf{CMon}(can\text{-}\mathcal{S}mod_M)$ (because every cancellative $M$-semimodule $(A,\eta)$ is an internal monoid in $\mathcal{S}mod_M$ under the monoid operation $+\colon A\times A\rightarrow A$ on $A$ itself) one concludes with the following:
\begin{corollary}
For any $(A,\eta)\in can\text{-}\mathcal{S}\mathrm{mod}_M,$ the fibres $D_1^{-1}(A,\eta)$ and $d_1^{-1}(A,\eta)$ admit a natural structure of symmetric monoidal category.
\end{corollary}

\end{document}